\documentclass[12pt,reqno]{amsart}

\usepackage{amsmath,amssymb,ifthen}
\usepackage{fullpage}
\usepackage{graphicx,psfrag,subfigure}
\usepackage{color}
\usepackage{moreverb}
\usepackage{amsthm}
\usepackage{mathrsfs}
\usepackage{hhline}
\usepackage{bbm}
\usepackage[english]{babel}
\usepackage{tikz}
\usepackage{pgfplots}
\usepackage[utf8]{inputenc}
\usepackage{amsfonts}
\usepackage{enumerate}
\usepackage{hyperref}

\newcommand{\const}[1]{C_{\text{\rm#1}}}

\newcommand{\set}[2]{\big\{#1\,:\,#2\big\}}

\newcommand{\dual}[3][]{#1\langle#2\,,\,#3#1\rangle}
\newcommand{\norm}[3][]{#1\|#2#1\|_{#3}}

\newcommand\R{\mathbb{R}}

\newcommand\LL{\mathcal L}

\newcommand\OO{{\mathcal O}}
\newcommand\PP{\mathcal P}

\newcommand\TT{\mathcal T}

\numberwithin{equation}{section}
\numberwithin{figure}{section}
\newtheorem{theorem}{Theorem}[section]

\newtheorem{lemma}[theorem]{Lemma}

\newtheorem{remark}[theorem]{Remark}

\newcommand*\patchAmsMathEnvironmentForLineno[1]{%
  \expandafter\let\csname old#1\expandafter\endcsname\csname #1\endcsname
  \expandafter\let\csname oldend#1\expandafter\endcsname\csname end#1\endcsname
  \renewenvironment{#1}%
     {\linenomath\csname old#1\endcsname}%
     {\csname oldend#1\endcsname\endlinenomath}}%
\newcommand*\patchBothAmsMathEnvironmentsForLineno[1]{%
  \patchAmsMathEnvironmentForLineno{#1}%
  \patchAmsMathEnvironmentForLineno{#1*}}%
\AtBeginDocument{%
\patchBothAmsMathEnvironmentsForLineno{equation}%
\patchBothAmsMathEnvironmentsForLineno{align}%
\patchBothAmsMathEnvironmentsForLineno{flalign}%
\patchBothAmsMathEnvironmentsForLineno{alignat}%
\patchBothAmsMathEnvironmentsForLineno{gather}%
\patchBothAmsMathEnvironmentsForLineno{multline}%
}
\usepackage[mathlines]{lineno}

\usepackage{bm}

\DeclareMathOperator{\divv}{div}
\DeclareMathOperator{\essinf}{ess\,inf}
\DeclareMathOperator*{\argmin}{argmin}
\DeclareMathOperator{\clos}{clos}
\DeclareMathOperator{\ran}{ran}
\newcommand{\cL}{\mathcal L}

\newcommand{\be}{\begin{equation}}
\newcommand{\ee}{\end{equation}}

\DeclareSymbolFont{bbold}{U}{bbold}{m}{n}
\DeclareSymbolFontAlphabet{\mathbbold}{bbold}
\newcommand{\ind}{\mathbbold{1}}

\usetikzlibrary{external}
\usetikzlibrary{calc,math}
\usepackage{tikz-3dplot}

\renewcommand{\d}{\,{\rm d}} 
 
\newcommand{\x}{{\bf x}}
\newcommand{\mmu}{{\boldsymbol \mu}}

\title{Applications of a space-time FOSLS formulation \\for parabolic PDEs}

\author{Gregor Gantner} 
\address{Institute of Analysis and Scientific Computing, TU Wien, Wiedner Hauptstra{\ss}e 8-10, 1040 Vienna, Austria.}
\email{gregor.gantner@asc.tuwien.ac.at}

\author{Rob Stevenson}
\address{Korteweg-de Vries (KdV) Institute for Mathematics, University of Amsterdam, P.O. Box 94248, 1090 GE Amsterdam, The Netherlands.}
\email{rob.p.stevenson@gmail.com}

\thanks{The first author has been supported by the Austrian Science Fund (FWF) under grant J4379-N. The second author has been supported by NSF Grant DMS 172029}

\begin{document}

\begin{abstract}
In this work, we show that the space-time first-order system least-squares (FOSLS) formulation~[F\"uhrer, Karkulik, \emph{Comput.\ Math.\ Appl.}~92 (2021)] for the heat equation
and its recent generalization~[Gantner, Stevenson, \emph{ESAIM Math.\ Model.\ Numer.\ Anal.}~55 (2021)] to arbitrary second-order parabolic PDEs can be used to efficiently solve parameter-dependent problems, optimal control problems, and problems on time-dependent spatial domains.
\end{abstract}

\date{\today}
\maketitle


\section{Introduction}
It is well-known, e.g., from \cite{wloka87,ss09}, that the operator corresponding to the heat equation $\partial_t u -\Delta_{\bf x} u = f$, $u(0,\cdot)=u_0$ on a time-space cylinder $I\times \Omega$, where $I:=(0,T)$ and $\Omega \subset \R^d$,  with homogeneous Dirichlet boundary conditions, is a boundedly invertible linear mapping between $X$ and $Y' \times L_2(\Omega)$, where $X:=L_2(I;H^1_0(\Omega)) \cap H^1(I;H^{-1}(\Omega))$ and $Y:=L_2(I;H^1_0(\Omega))$.
For source term $f \in L_2(I \times \Omega)$ and initial data $u_0\in L_2(\Omega)$, the recent work~\cite{fk21} has proven that 
\begin{align*}
\argmin_{{\bf u}=(u_1,{\bf u}_2) \in U} \|\divv {\bf u}_2-f\|_{L_2(I\times\Omega)}^2+\|{\bf u}_2 +\nabla_{\bf x} u_1\|^2_{L_2(I\times\Omega)^d}+\|u(0,\cdot)-u_0\|^2_{L_2(\Omega)},
\end{align*}
where $U:=\{{\bf u}\in X \times L_2(I\times \Omega)^d\colon\divv {\bf u} \in L_2(I\times\Omega)\}$ equipped with the graph norm,
is a well-posed first-order system least-squares (FOSLS) formulation for the pair of the solution $u=u_1$ and (minus) its spatial gradient $-\nabla_{\bf x} u={\bf u}_2$.
This formulation can already be found in \cite{bg09} without a proof of its well-posedness though.
In \cite{gs21}, we have generalized the result to second-order parabolic PDEs with arbitrary Dirichlet and/or Neumann boundary conditions (in the case of inhomogeneous boundary conditions appending a squared norm of a boundary residual, which is not of $L_2$-type, to the least-squares functional). 
In particular, we have shown that source terms $f\not\in L_2(I\times \Omega)$ are covered and $X$ in the definition of $U$ may be replaced by $Y$. 
We also mention our recent generalization to the instationary Stokes problem with so-called slip boundary conditions~\cite{gs22}.

Compared to other space-time discretization approaches such as \cite{andreev13,steinbach15,lmn16,sw21a,sw21b}, the FOSLS formulation has the major advantage that it results in a symmetric, and, w.r.t.~a mesh-independent norm, bounded and coercive bilinear form
so that the Galerkin approximation from \emph{any} conforming trial space is a quasi-best approximation from that space. 
At least for homogeneous boundary conditions, the minimization is w.r.t.~$L_2$-norms, so that the arising stiffness matrix is sparse and can be easily computed. 
Moreover, the least-squares functional provides a reliable and efficient a posteriori estimator for the error in the $U$-norm.
One disadvantage of the FOSLS method from \cite{fk21,gs21} is that the latter norm for the error in the pair $(u,-\nabla_{\bf x} u)$ appears to be considerably stronger than the $X$-norm for the error in $u$. 
Indeed, for standard Lagrange finite element spaces applied to non-smooth solutions, e.g.,~as those that result from a discontinuity in the transition of initial and boundary data, relatively low convergence rates are reported in \cite{fk21}, even when adaptive refinement is employed. 
This issue shall be addressed in our future work~\cite{gs22+} by constructing more suitable trial spaces. 
 
In the present work, we exploit the aforementioned advantages of the FOSLS method to efficiently solve parameter-dependent problems, optimal control problems, and problems on time-dependent spatial domains.

\subsection{Parameter-dependent problems}
We consider a reduced basis method, where in a potentially expensive \emph{offline phase} a basis of (highly accurate approximations of the) solutions corresponding to a suitable finite subset of the parameter space is computed so that in the subsequent \emph{online phase} the solution for arbitrary given parameters can be efficiently approximated from the spanned reduced basis space.

Compared to reduced basis methods based on time stepping and proper orthogonal decomposition (POD), e.g., \cite{ho08,ekp11,haasdonk17}, reduced basis methods based on simultaneous space-time discretization, e.g., \cite{up14,yano14,ypu14}, yield not only a dimension reduction in space but in space-time; see the comparison in \cite{gmu17}. Consequently, in any case the resulting online phase can be expected to be more efficient.

As already mentioned, compared to other space-time discretization approaches, the FOSLS formulation has the advantage that it results in a coercive bilinear form. In this setting the strongest theoretical results for the reduced basis method are available (see, e.g., \cite{haasdonk17}), and the construction of the reduced basis does not need to be accompanied by the construction of a corresponding test space that yields inf-sup stability.

Finally, the reliable and efficient a posteriori error estimator of the FOSLS formulation is beneficial for both building the reduced basis and for certification of the obtained approximate solution in the online phase. 


\subsection{Optimal control problems}
We consider a large class of optimal control problems constrained by second-order parabolic PDEs. 
This includes \emph{controls} of the source term in $L_2(I\times \Omega)$ or in $Y'$ 
 and of the initial data in $L_2(\Omega)$ as well as desired \emph{observations} of the PDE solution in $L_2(I\times\Omega)$, its spatial gradient in $L_2(I\times\Omega)^d$, or its restriction to the final time point in $L_2(\Omega)$. 

It is well-known that optimal control problems can be equivalently formulated as saddle-point problem; see, e.g., \cite{lions68}.
The saddle-point formulation is a coupled system of the primal equation and some adjoint equation. 
When solving it iteratively, e.g., by a gradient descent method  \cite{troltzsch10,hpuu08}, the approximate solutions are required over the whole space-time cylinder $I\times \Omega$ and not just at the final time point. 
This nullifies the potential advantage of time-stepping methods such as \cite{mv07,mv08}, which in the case of a forward problem
only need to store the approximation at the current time step. 

Instead space-time methods such as \cite{ghz12,lsty21a,lsty21b,ls21}  aim to solve the saddle-point problem at once on some mesh of the space-time cylinder $I\times \Omega$. 
All of the mentioned works consider optimal control problems constrained by the heat equation with control of the source term and desired observation of the PDE solution itself. 
In contrast to time-stepping methods, all of them allow for locally adapted meshes because the involved bilinear forms are coercive or inf-sup stable. 
However, only \cite{ghz12}, which reformulates the problem as a system of fourth order in space and second order in time requiring additional regularity, yields quasi-optimal approximations in the usual sense with the same norm on both sides, as for the others the norms in which the bilinear forms are continuous differ from the norms in which they are coercive or inf-sup stable.

In our approach, we consider the saddle-point problem when the parabolic PDE is replaced by the equivalent FOSLS. 
Coercivity of the latter implies that this saddle-point problem is for \emph{arbitrary} discrete subspaces uniformly inf-sup stable and continuous w.r.t.\ the same norm so that quasi-optimality in the usual sense is valid. 
We also provide one reliable and one efficient a posteriori error estimator for our method. 
Other estimators have been derived in \cite{ghz12,ls21} for the respective methods. 

\subsection{Time-dependent domains}
Finally, we consider second-order parabolic PDEs on time-dependent domains. 
While this poses a technical difficulty for time-stepping methods~\cite{gs17,fr17,lo19,sg20,fs22}, in principle, space-time methods~\cite{lehrenfeld15,hlz16,lmn16,moore18,bht22} can be readily applied because the overall space-time domain remains fixed. 

We show that this is indeed the case for the FOSLS method~\cite{fk21,gs21} verifying its well-posedness. 
The analysis requires minimal regularity assumptions on the mapping describing the motion of the spatial domain throughout time.
However, we emphasize that our method ultimately only requires a mesh of the space-time domain.
This is in contrast to the so-called arbitrary Lagrangian-Eulerian (ALE) time-stepping methods~\cite{shd01,gs17,sg20}, which solve the problem on the transformed time-independent domain and thus explicitly involve the mapping. 
Moreover, our method again allows for arbitrary discrete trial spaces and immediately provides an a reliable and efficient a posteriori error estimator.

\subsection{Outline}\label{sec:outline}
In Section~\ref{sec:preliminaries}, we fix the general notation (Section~\ref{sec:notation}) and recall the formulation of general second-order parabolic PDEs as a first-order system (Section~\ref{sec:fosls}). 
In Section~\ref{sec:parameter problems}, we introduce a corresponding reduced basis method for parameter-dependent problems (Section~\ref{sec:reduced basis}), discuss the generation of a suitable basis (Section~\ref{sec:greedy algorithm}), and provide a numerical experiment (Section~\ref{sec:reduced 1d}).
In Section~\ref{sec:optimal control}, we show that the formulation of optimal control problems constrained by the first-order system as a saddle-point problem is uniformly stable for arbitrary trial spaces (Section~\ref{sec:saddle point}), derive an optimality system of PDEs for a certain class of optimal control problems (Section~\ref{sec:optimality system}), which is then used to derive a posteriori error estimators (Section~\ref{sec:a posteriori}), and provide two numerical experiments (Section~\ref{sec:numerics optimal control}), where we also demonstrate optimal convergence rate for piecewise linear trial functions if both the optimal control and the corresponding state are sufficiently smooth. 
In Section~\ref{sec: time-dependent domains}, we show that the formulation of general second-order PDEs on time-dependent domains as first-order system induces, as for time-independent domains, a linear isomorphism and provide two numerical experiments (Section~\ref{sec:move numerics}). 
Finally, we draw some conclusion in Section~\ref{sec:conclusion}.

\section{Preliminaries}\label{sec:preliminaries}

\subsection{General notation}\label{sec:notation}
In this work, by $C \lesssim D$ we will mean that $C\ge0$ can be bounded by a multiple of $D\ge0$, independently of parameters on which $C$ and $ D$ may depend. 
Obviously, $C \gtrsim D$ is defined as $C \lesssim D$, and $C\eqsim D$ as $C\lesssim D$ and $C \gtrsim D$.

For normed linear spaces $E$ and $F$, we will denote by $\LL(E,F)$ the normed linear space of bounded linear mappings $E \rightarrow F$. 
For simplicity only, we exclusively consider linear spaces over the scalar field $\R$.

\subsection{Formulation of parabolic PDEs as first-order system}\label{sec:fosls}
Let $\Omega\subset\R^d$, $d\ge 1$, be a Lipschitz domain with boundary $\Gamma:=\partial\Omega$, and $T>0$ a given end time point with corresponding time interval $I:=(0,T)$. 
We abbreviate the space-time cylinder $Q:=I\times\Omega$ with lateral boundary $\Sigma:=I\times\Gamma$. 
We consider the following parabolic PDE with homogeneous Dirichlet boundary conditions

\begin{equation}\label{eq:parabolic pde}
\begin{array}{rcll}
 \partial_t u - \divv_\x
 ({\bf A} \nabla_\x u)+{\bf b}\cdot \nabla_\x u + cu & = & f & \text{ in } Q,\\
 u & = & 0 & \text{ on }\Sigma,\\
 u(0,\cdot) & = & u_0 & \text{ on }\Omega.
\end{array}
\end{equation}
Here, we will require that ${\bf A}={\bf A}^\top\in L_\infty(Q)^{d\times d}$ is uniformly positive, ${\bf b}\in L_\infty(Q)^d$, and $c\in L_\infty(Q)$.
For unique existence of a (weak) solution $u$ of \eqref{eq:parabolic pde}, we recall the following theorem from, e.g., \cite[Theorem~5.1]{ss09}.
\begin{theorem}\label{thm:ss09}
With $X:= L_2(I;H_0^1(\Omega)) \cap H^1(I;H^{-1}(\Omega))$, $Y:= L_2(I;H_0^1(\Omega))$, 
\begin{align*}
 (Bu)(v) := \int_Q \partial_t u \,v+ ({\bf A}\nabla_\x u)\cdot\nabla_\x v + {\bf b}\cdot \nabla_\x u\,v + cuv \d \x \d t
\end{align*}
and $\gamma_0(u):= u(0,\cdot)$ for all $u\in X,v\in Y$, the mapping
\begin{align*}
 \begin{pmatrix}B\\\gamma_0\end{pmatrix}: X\to Y'\times L_2(\Omega), 
\end{align*}
is a linear isomorphism. In particular, there exists a unique $u \in X$ such that
\begin{equation}\label{eq:second-order system}
 \begin{pmatrix}B\\\gamma_0\end{pmatrix} u=\begin{pmatrix}f\\ u_0\end{pmatrix}.
 \end{equation}
\end{theorem}

Any $f \in Y'=L_2(I;H^{-1}(\Omega))$ can be written as
\begin{equation} \label{split}
f =f_1+\divv_{\bf x} {\bf f}_2,
\end{equation}
for some $f_1\in L_2(Q)$ and ${\bf f}_2 \in L_2(Q)^d$, where $\int_Q \divv_{\bf x} {\bf f}_2 \,v \d {\bf x}\d t:=-\int_{Q} {\bf f}_2\cdot\nabla_{\bf x} v \d {\bf x} \d t$
for $v \in Y$. With such a decomposition, and ${\bf u}=(u_1,{\bf u}_2)\colon Q\rightarrow \R \times \R^d$, \eqref{eq:second-order system} is equivalent to the first-order system\footnote{The system $\left(\begin{array}{@{}c@{}} \divv {\bf u} - {\bf b}\cdot {\bf A}^{-1} {\bf u}_2 + cu_1 \\ {\bf u}_2 + {\bf A}\nabla_\x u_1 \\  u_1(0,\cdot)\end{array} \right)   =  \left(\begin{array}{@{}c@{}}  f_1 + {\bf b}\cdot {\bf A}^{-1} {\bf f}_2 \\ -{\bf f}_2\\ u_0 \end{array} \right)$ studied in \cite{gs21} leads to \eqref{eq:first-order system} by pre-multiplication with 
$\left(\begin{array}{@{}ccc@{}} I & {\bf b}\cdot {\bf A}^{-1} & 0 \\ 0 & - I & 0 \\ 0 & 0 & I\end{array} \right)$, which is a linear isomorphism on the space $L$ defined in \eqref{eq:L}}.
\begin{align}\label{eq:first-order system}
 G{\bf u}:=  \left(\begin{array}{@{}c@{}} \divv {\bf u} + {\bf b}\cdot \nabla_\x u_1 + cu_1 \\ -{\bf u}_2 - {\bf A}\nabla_\x u_1 \\  u_1(0,\cdot)\end{array} \right) 
  =  \left(\begin{array}{@{}c@{}}  f_1  \\ {\bf f}_2\\ u_0 \end{array} \right)=:{\bf f},\quad u_1|_\Sigma = 0,
\end{align}
as is shown in the next theorem.

\begin{theorem}[{\cite[Theorem~2.3 and Proposition~2.5]{gs21}}]\label{thm:fosls}
The operator $G$ is a linear isomorphism from the space
\begin{align}\label{eq:U}
U := \{{\bf u}=(u_1,{\bf u}_2) \in Y\times L_2(Q)^d\colon {\bf u} \in H(\divv;Q)\} 
\end{align}
(equipped with the corresponding graph norm) to the space 
\begin{align}\label{eq:L}
 L:=L_2(Q)\times L_2(Q)^d \times L_2(\Omega).\footnotemark
\end{align}
\footnotetext{In the original result from~\cite{fk21} for the heat equation, it was shown that $G$ is a linear isomorphism between $U$ and its range in $L$.}%
If $u\in X$ solves \eqref{eq:second-order system}, ${\bf u} = (u,-{\bf A}\nabla_\x u - {\bf f}_2)\in U$ solves \eqref{eq:first-order system}. Conversely, if ${\bf u}=(u_1,{\bf u}_2) \in U$ solves \eqref{eq:first-order system}, then $u=u_1$ solves \eqref{eq:second-order system}.
\end{theorem}

Notice that $G{\bf u}={\bf f}$ is equivalent to the variational problem
\begin{align*}
\langle G{\bf u}, G{\bf v}\rangle_L=\langle {\bf f},G{\bf v}\rangle_{L} \quad\text{for all } {\bf v} \in U.
\end{align*}
Since the bilinear form at the left-hand side is bounded, symmetric and coercive, it provides the ideal setting for the application of Galerkin discretizations. 
The Galerkin solution from the employed trial space is a quasi-best approximation to ${\bf u}$ w.r.t.~$\|\cdot\|_U$.
For any approximation $\widetilde{\bf u}=(\widetilde{u}_1,\widetilde{\bf u}_2)$ of ${\bf u}$ we have the computable a posteriori error estimator $\|{\bf u}-\widetilde{\bf u}\|_U \eqsim \|{\bf f}-G\widetilde{\bf u}\|_L$.
Moreover, the following lemma shows that $\|u_1-\widetilde{u}_1\|_X \lesssim \|{\bf u}-\widetilde{\bf u}\|_U$.

\begin{lemma}[{\cite[Lemma~2.2]{gs21}}]\label{lem:U2X}
The mapping ${\bf u}\mapsto u_1$ belongs to $\LL(U,X)$.
\end{lemma}

Finally in this section, we provide more information about the splitting \eqref{split}.
\begin{remark} \label{rem:splitting}
For any $f \in Y'$, the unique solution of 
\begin{align*}
\argmin_{\{(f_1,{\bf f}_2) \in L_2(Q) \times L_2(Q)^d\colon f_1+\divv_{\bf x} {\bf f}_2=f\}} \|f_1\|_{L_2(Q)}^2+\|{\bf f}_2\|_{L_2(Q)^d}^2
\end{align*}
is given by $(f_1,{\bf f}_2)=(w,-\nabla_{\bf x} w)$ where $\langle w,v\rangle_Y:=\langle w,v\rangle_{L_2(Q)} + \langle \nabla_\x w,\nabla_\x v\rangle_{L_2(Q)}=f(v)$ for all $v \in Y$, i.e., $w \in Y$ is the Riesz lift of $f$, and
\begin{align*}
 \|f_1\|_{L_2(Q)}^2+\|{\bf f}_2\|_{L_2(Q)^d}^2=\|f\|^2_{Y'}.
\end{align*}
Indeed, the last property is a consequence of $\|w\|_Y=\|f\|_{Y'}$, whereas for arbitrary $(f_1,{\bf f}_2) \in L_2(Q) \times L_2(Q)^d$ with $f_1+\divv_{\bf x} {\bf f}_2=f$, 
$\|f\|^2_{Y'} \leq \|f_1\|_{L_2(Q)}^2+\|{\bf f}_2\|_{L_2(Q)^d}^2$ follows from
$f(v)=\langle f_1,v\rangle_{L_2(Q)}-\langle {\bf f}_2,\nabla_{\bf x} v\rangle_{L_2(Q)^d} \leq \sqrt{\|f_1\|_{L_2(Q)}^2+\|{\bf f}_2\|_{L_2(Q)^d}^2}\,\|v\|_Y$.

Finally, notice that if $f =f_1+\divv_{\bf x} {\bf f}_2$ with $ \|f_1\|_{L_2(Q)}^2+\|{\bf f}_2\|_{L_2(Q)^d}^2 \eqsim \|f\|_{Y'}^2$, then for the solutions ${\bf u}$ of \eqref{eq:first-order system} and $u$ of \eqref{eq:second-order system},
it holds that
$\|{\bf u}\|^2_U \eqsim \|{\bf f}\|^2_L \eqsim \|f\|^2_{Y'}+\|u_0\|_{L_2(\Omega)}^2 \eqsim \|u\|_X^2$.
\end{remark}

\section{Parameter-dependent problems}\label{sec:parameter problems}

In this section, we consider coefficients ${\bf A}={\bf A}[\mmu], {\bf b}={\bf b}[\mmu], c=c[\mmu]$ and right-hand side 
${\bf f}[\mmu]=(f_1[\mmu],{\bf f}_2[\mmu],u_0[\mmu])$ that depend additionally on a tuple of parameters
${\bf \mmu}$ in a set $\PP\subset \R^p$. 
We assume that the coefficients and the right hand sides are \emph{parameter-separable} in the sense that
\begin{equation}\label{eq:parameter separable}
\begin{array}{rcl rcl rcl}
 {\bf A}[\mmu] &=& \sum_{q = 1}^{n_{\bf A}} \theta_q^{\bf A}(\mmu) {\bf A}_q, \quad
 &{\bf b}[\mmu] &=& \sum_{q = 1}^{n_{\bf b}} \theta_q^{\bf b}(\mmu) {\bf b}_q, \quad
 &c[\mmu] &=& \sum_{q = 1}^{n_c} \theta_q^c(\mmu) c_q,
 \\
 f_1[\mmu] &=& \sum_{q = 1}^{n_{f_1}} \theta_q^{f_1}(\mmu) f_{1,q}, \quad
 &{\bf f}_2[\mmu] &=& \sum_{q = 1}^{n_{{\bf f}_2}} \theta_q^{\bf f_2}(\mmu) {\bf f}_{2,q}, \quad
 &u_0[\mmu] &=& \sum_{q = 1}^{n_{u_0}} \theta_q^{u_0}(\mmu) u_{0,q}
\end{array}
\end{equation}
for some functions $\theta_q^{(\cdot)}:\PP\to \R$ and ${\bf A}_q \in L_\infty(\Omega)^{d\times d}, {\bf b}_q\in L_\infty(\Omega)^d, c_{q}\in L_\infty(\Omega)$ and $f_{1,q}\in L_2(Q), {\bf f}_{2,q} \in L_2(Q)^d, u_{0,q}\in L_2(\Omega)$. 
Let $F=F[\mmu]\in U'$  be some given quantity of interest that is also parameter-separable, i.e., $F[\mmu] = \sum_{q = 1}^{n_F} \theta_q^F(\mmu) F_q$ with $\theta_q^F:\PP\to \R$ and $F_q\in U'$. 
After a possibly computationally expensive \emph{offline phase}, we want to be able to instantly compute an approximation of $F[\mmu]({\bf u}[\mmu])$ for different $\mmu\in\PP$ in the so-called \emph{online phase}. 

\subsection{Reduced basis method}\label{sec:reduced basis}
Given some parameter $\mmu$, the idea of the reduced basis method is to compute an approximation ${\bf u}^N[\mmu]$ of ${\bf u}[\mmu]$ from the (low-dimensional) span of some \emph{snapshots} $\{{\bf u}[\mmu^{(1)}],\dots,{\bf u}[\mmu^{(N)}]\}$.

Instead of ${\bf u}[\mmu^{(i)}]$, very accurate approximations ${\bf u}^\delta[\mmu^{(i)}]$ thereof are computed in the offline phase. We will choose ${\bf u}^\delta[\mmu^{(i)}]$ as the best approximation to ${\bf u}[\mmu^{(i)}]$ w.r.t.~$\|G[\mmu^{(i)}] (\cdot)\|_L$ from some high-dimensional subspace $U^\delta\subset U$, i.e.,
as the solution of the Galerkin system $\langle G[\mmu^{(i)}] {\bf u}^\delta[\mmu^{(i)}]-{\bf f}[\mmu^{(i)}], G[\mmu^{(i)}] {\bf v}^\delta \rangle_L=0$ for all ${\bf v}^\delta \in U^\delta$.
One easily checks that the parameter-separability~\eqref{eq:parameter separable} of the coefficients and the right-hand sides implies parameter-separability of the bilinear form  
\begin{align*}
 \dual{G[\mmu](\cdot)}{G[\mmu](\cdot)}_L = \sum_{q=1}^{n_b} \theta_q^b(\mmu) b_q(\cdot,\cdot), 
\end{align*}
the linear form 
\begin{align*}
 \dual{{\bf f}[\mmu]}{G[\mmu](\cdot)}_L
 = \sum_{q=1}^{n_l} \theta_q^l(\mmu) l_q(\cdot),
\end{align*}
and the squared norm 
\begin{align*}
 \norm{G[\mmu]({\bf u})}{L}^2 
 = \norm{{\bf f}[\mmu]}{L}^2
 = \sum_{q=1}^{n_s} \theta_q^s(\mmu), 
\end{align*}
where again $\theta_q^{(\cdot)}:\PP\to\R$, and each $b_q:U\times U\to \R$ is a continuous bilinear form and each $l_q:U\to\R$ is a continuous linear form. 
Besides the functions $ {\bf u}^{\delta}[\mmu^{(i)}]$, which form the reduced basis, in the offline phase we further compute the matrices and vectors
\begin{align*}
 \big(b_q( {\bf u}^\delta[\mmu^{(j)}], {\bf u}^\delta[\mmu^{(i)}])\big)_{i,j=1}^N, \quad \big(l_q( {\bf u}^\delta[\mmu^{(i)}])\big)_{i=1}^N, \quad \big(F_q( {\bf u}^\delta[\mmu^{(i)}])\big)_{i=1}^N,
\end{align*}

In the online phase, we compute ${\bf u}^N[\mmu]$ as the best approximation to ${\bf u}[\mmu]$ w.r.t.~\mbox{$\|G[\mmu] (\cdot)\|_L$}
from the span of the reduced basis, i.e., as the solution of a low-dimensional Galerkin system.
Assuming that the snapshots ${\bf u}^\delta[\mmu^{(1)}],$ $\dots$, ${\bf u}^\delta[\mmu^{(N)}]$ are linearly independent, the corresponding coefficient vector ${\bf c}^N[\mmu]$ of ${\bf u}^N[\mmu]$ with respect to this basis is just given by 
\begin{align*}
 {\bf c}^N[\mmu] = \Big(\sum_{q=1}^{n_b} \theta_q^b(\mu) \big(b_q( {\bf u}^{\delta}[\mmu^{(j)}], {\bf u}^{\delta}[\mmu^{(i)}])\big)_{i,j=1}^N\Big)^{-1}
  \Big(\sum_{q=1}^{n_l} \theta_q^l(\mu)\big(l_q( {\bf u}^{\delta}[\mmu^{(i)}])\big)_{i=1}^N\Big).
\end{align*}
We stress that the computational effort to solve this linear system depends only on $N$ and \emph{not} on ${\rm dim}(U^\delta)$ as the involved matrices and vectors have been already computed in the offline phase. 
Then, the quantity of interest 
is given by 
\begin{align*}
 F[\mmu]({\bf u}^N[\mmu]) =\sum_{q=1}^{n_F} \theta_q^F(\mmu) \Big( \big(F_q( {\bf u}^\delta[\mmu^{(i)}])\big)_{i=1}^N \cdot {\bf c}^N[\mmu]\Big).
\end{align*}
Again, the involved vectors $F_q( {\bf u}^\delta[\mmu^{(i)}])\big)_{i=1}^N$ have been precomputed in the offline phase. 
Finally, we can even instantly estimate the discretization error in the online phase by
\begin{align*}
\norm{{\bf u}[\mmu] - {\bf u}^N[\mmu]}{U}^2 &\eqsim \norm{G[\mmu] ({\bf u}[\mmu]) - G[\mmu]({\bf u}^N[\mmu])}{L}^2
 \\
 &= \norm{G[\mmu]({\bf u}[\mmu])}{L}^2 - \dual{G[\mmu]({\bf u}[\mmu])}{G[\mmu]({\bf u}^N[\mmu])}_L 
 \\
 &= \sum_{q=1}^{n_s} \theta_q^s(\mmu) -   \sum_{q=1}^{n_l} \theta_q^l(\mmu)\Big(\big(l_q( {\bf u}^{\delta}[\mmu^{(i)}])\big)_{i=1}^N \cdot {\bf c}^N[\mmu]\Big),
\end{align*}
The constants hidden in the $\eqsim$-symbol depend only on the $L_\infty$-norms of the coefficients ${\bf A}[\mmu],{\bf b}[\mmu],c[\mmu]$ and the smallest eigenvalue of ${\bf A}[\mmu]$.

\subsection{Basis generation}\label{sec:greedy algorithm}
It remains to explain how to determine a suitable reduced basis $\{{\bf u}^\delta[\mmu^{(1)}],$ $\dots$, ${\bf u}^\delta[\mmu^{(N)}]\}$. 
Given some sufficiently large training set $\PP_{\rm train}\subseteq\PP$ and some tolerance $\epsilon_{\rm tol}>0$, 
we employ a greedy algorithm that starting with ${\bf u}^0[\mmu] := 0$, $\mmu\in\PP$, iteratively adds the snapshot
\begin{align*}
 {\bf u}^\delta\big[\operatorname*{argmax}\limits_{\mmu\in\PP_{\rm train}}\norm{G[\mmu] ({\bf u}[\mmu]) - G[\mmu]({\bf u}^N[\mmu])}{L}\big]
\end{align*}
to the reduced basis and increments $N$ by one until 
\begin{align*}
 \max_{\mmu\in\PP_{\rm train}} \norm{G[\mmu] ({\bf u}[\mmu]) - G[\mmu]({\bf u}^N[\mmu])}{L}\le\epsilon_{\rm tol}.
\end{align*}
Note that this procedure terminates at most after $\#\PP_{\rm train}$ steps and provides indeed a basis if the discretization error in $U^\delta$ is negligible. 
For possible choices of the training set $\PP_{\rm train}$, we refer, e.g., to~\cite[Remark~2.44]{haasdonk17}. 

\subsection{Numerical experiment}\label{sec:reduced 1d}
We consider the example from \cite{gmu17}: 
Let $\Omega=(0,1)$ and $T:=0.3$. 
We consider the parameter set $\PP:=[0.5,1.5] \times [0,1] \times [0,1]\subset \R^3$ with ${\bf A}[\mmu] := \mu_1$, ${\bf b}[\mmu] :=  \mu_2$, and $c[\mmu] = \mu_3$. 
Moreover, we choose the right-hand sides $f_1$, ${\bf f}_2$, and $u_0$ independently of the parameters with 
\begin{align*}
 f_1(t,x):=\sin(2\pi x)\big((4\pi^2 + 0.5) \cos(4\pi t) - 4\pi\sin(4\pi t)\big) + \pi \cos(2\pi x)\cos(4\pi t)
\end{align*}
on the space-time cylinder $Q=(0,1)^2$, ${\bf f}_2 := 0$, and $u_0(x) := \sin(2\pi x)$ on $\Omega$, 
which corresponds to the solution $u[(1,0.5,0.5)](t,x) = \sin(2\pi x)\cos(4\pi t)$. 

\begin{center}
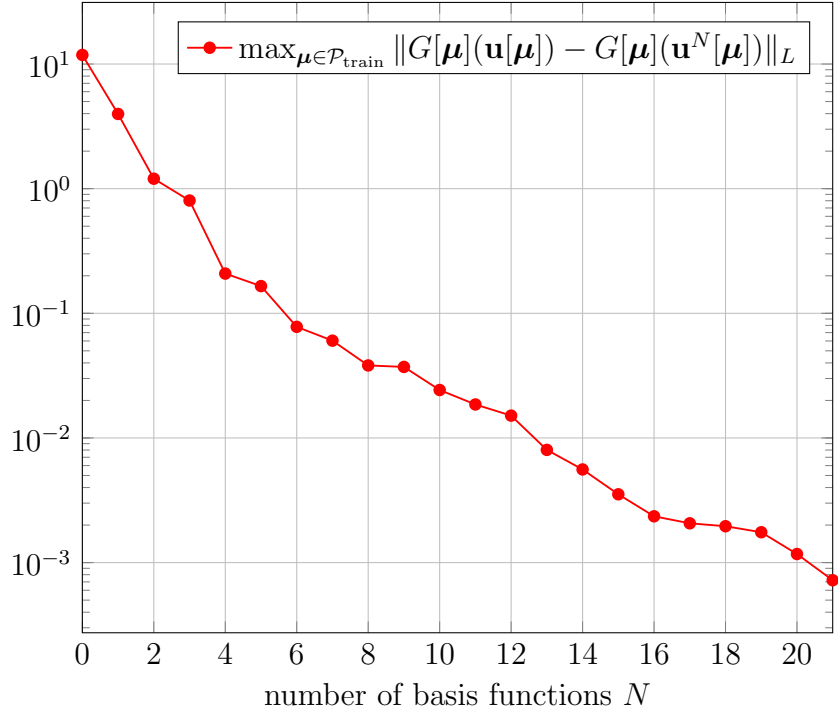
\begin{figure}
\begin{tikzpicture}
\begin{semilogyaxis}[   
width = 0.7\textwidth,
clip = true,
xmin=0,xmax=21,
xlabel= {number of basis functions $N$},
grid = major,
legend pos = north east,
legend entries = {$ \max_{\mmu\in\PP_{\rm train}} \norm{G[\mmu] ({\bf u}[\mmu]) - G[\mmu]({\bf u}^N[\mmu])}{L}$\\} 
]
\addplot[line width = 0.25mm,color=red,mark=*] table[x=N,y=maxtrain, col sep=comma] {data/stred_off_D1_order3_nx64_n17.csv};
\end{semilogyaxis}
\end{tikzpicture}
\caption{\label{fig:reduced off 1d} Approximation error of greedy algorithm for reduced basis method of Section~\ref{sec:reduced 1d}.}
\end{figure}
\end{center}

\begin{center}
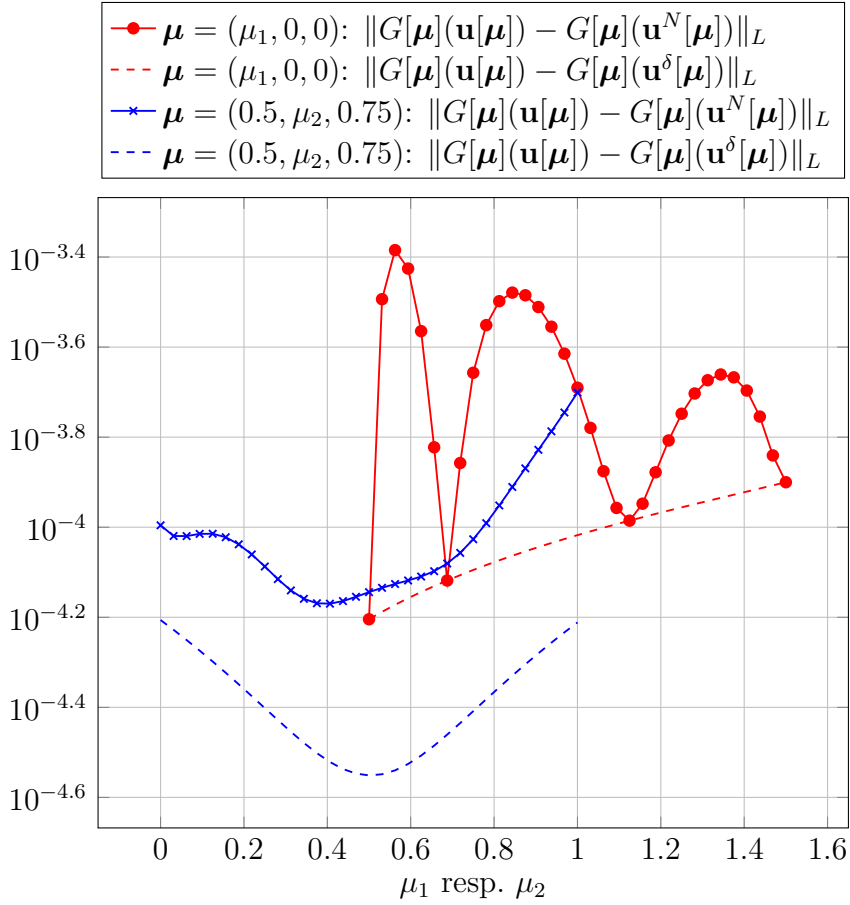
\begin{figure}
\begin{tikzpicture}
\begin{semilogyaxis}[
width = 0.7\textwidth,
clip = true,
xlabel= {$\mu_1$ resp.\ $\mu_2$},
grid = major,
legend entries = {$\mmu = (\mu_1,0,0)$: $\norm{G[\mmu] ({\bf u}[\mmu]) - G[\mmu]({\bf u}^N[\mmu])}{L}$\\ $\mmu = (\mu_1,0,0)$: $\norm{G[\mmu] ({\bf u}[\mmu]) - G[\mmu]({\bf u}^\delta[\mmu])}{L}$ \\%
$\mmu = (0.5,\mu_2,0.75)$: $\norm{G[\mmu] ({\bf u}[\mmu]) - G[\mmu]({\bf u}^N[\mmu])}{L}$\\$\mmu = (0.5,\mu_2,0.75)$: $\norm{G[\mmu] ({\bf u}[\mmu]) - G[\mmu]({\bf u}^\delta[\mmu])}{L}$\\} ,
legend style={at={(0.5,1.03)},anchor=south},
legend cell align={left}
]
\addplot[line width = 0.25mm,color=red,mark=*] table[x=A,y=est_A, col sep=comma] {data/stred_on_D1_order3_nx64_n17.csv};
\addplot[dashed,line width = 0.25mm,color=red] table[x=A,y=err_A, col sep=comma] {data/stred_on_D1_order3_nx64_n17.csv};
\addplot[line width = 0.25mm,color=blue,mark=x] table[x=b,y=est_b, col sep=comma] {data/stred_on_D1_order3_nx64_n17.csv};
\addplot[dashed,line width = 0.25mm,color=blue] table[x=b,y=err_b, col sep=comma] {data/stred_on_D1_order3_nx64_n17.csv};
\end{semilogyaxis}
\end{tikzpicture}
\caption{\label{fig:reduced on 1d} Approximation errors at $\mmu = (\mu_1,0,0)$ with $\mu_1\in[0.5,1.5]$ (red) and $\mmu = (0.5,\mu_2,0.75)$ with $\mu_2\in[0,1]$  (blue) for reduced basis method of Section~\ref{sec:reduced 1d} with $N=21$.}
\end{figure}
\end{center}

We divide both $\Omega$ and $I=(0,T)$ into $2^6$ subintervals and choose $U^\delta$ to be the $2$-fold Cartesian product of continuous piecewise bi-cubic functions\footnote{The reason why we consider bi-cubic instead of bi-affine elements as in \cite{gmu17} is that we measure in a stronger norm but still want the approximation error in the space $U^\delta$ to be negligible.},
with the first coordinate space restricted by homogeneous Dirichlet boundary conditions on $\Sigma$.
The training set $\PP_{\rm train}$ is chosen as $17$ equidistantly distributed points in $\PP$ in each direction, and the tolerance $\epsilon_{\rm tol}$ is chosen as $10^{-3}$. 
Figure~\ref{fig:reduced off 1d} displays the approximation error $ \max_{\mmu\in\PP_{\rm train}} \norm{G[\mmu] ({\bf u}[\mmu]) - G[\mmu]({\bf u}^N[\mmu])}{L}$ throughout the greedy algorithm of the offline phase with the final $N=21$, where the $y$-axis is scaled logarithmically. 
As expected from~\cite{bcddpw11}, we observe exponential convergence. 
In Figure~\ref{fig:reduced on 1d}, we plot the discretization error $\norm{G[\mmu] ({\bf u}[\mmu]) - G[\mmu]({\bf u}^N[\mmu])}{L}$ of the reduced basis method applied for $\mmu$ in the test sets $\set{(\mu_1,0,0)}{\mu_1\in [0.5,1.5]}$ and $\set{(0.5,\mu_2,0.75)}{\mu_2\in [0,1]}$. 
For comparison, we also plot the best possible error $\norm{G[\mmu] ({\bf u}[\mmu]) - G[\mmu]({\bf u}^\delta[\mmu])}{L}$ {that one could hope for with the reduced basis method. 
Although the greedy algorithm only guarantees that these errors are below $\epsilon_{\rm tol}=10^{-3}$ on the training set $\PP_{\rm train}$, this bound 
appears to hold also true on the considered test sets. 

While a fair comparison to the space-time results of \cite{gmu17} is hard, we only mention that the considered errors, although measured in the stronger norm $\norm{G(\cdot)}{L}\simeq\norm{\cdot}{U}$, are of similar magnitude as the errors of \cite{gmu17} measured in (an approximation of) the norm $\norm{\cdot}{X}$.
The number of reduced basis functions to achieve an accuracy of $\epsilon_{\rm tol}=10^{-3}$ in \cite{gmu17} is given by $N=12$.
It is also notable that, in contrast to the estimator of \cite{gmu17}, which is based on the residual corresponding to (the first components of) ${\bf u}^\delta[\mmu] - {\bf u}^N[\mmu]$, the estimator considered here is  provably equivalent to the actual norm of the error $\|{\bf u}[\mmu] - {\bf u}^N[\mmu]\|_U$ and can also be efficiently computed in the online phase.

\section{Optimal control problems}\label{sec:optimal control}

Let ${\bf f}^\star=(f_1^\star,{\bf f}_2^\star,u_0^\star) \in L$ be fixed, and let $Z$ be a Hilbert space that is continuously embedded in $L$.
We consider \eqref{eq:first-order system} with ${\bf f}=(f_1,{\bf f}_2,u_0)={\bf f}^\star+{\bf z}$, where ${\bf z} \in Z$ is a \emph{control} variable.
The corresponding solution ${\bf u} \in U$ is then called \emph{state}. 
For $W$ being a further Hilbert space, $F\in\LL(U,W)$, \emph{desired observation} $w^\star\in W$, and 
 a parameter $\varrho>0$, we want to minimize the functional 
\begin{subequations}\label{eq:optimal control2}
\begin{align}
J({\bf u},{\bf z}) := \frac12 \norm{F{\bf u} - w^\star}{W}^2 + \frac{\varrho}{2}\norm{{\bf z}}{Z}^2
\end{align}
over
\begin{align}
 \set{({\bf u},{\bf z}) \in U \times Z}{G{\bf u}={\bf f}^\star+{\bf z}}. 
\end{align}
\end{subequations}

\begin{remark}
Let $F=\widetilde{F} \circ ({\bf u} \mapsto u_1)$ for some $\widetilde{F}\in\LL(X,W)$ and $Z=L_2(Q)\times L_2(Q)^d\times Z_3$ for some continuously embedded subspace $Z_3$ of $L_2(\Omega)$.
Then, \eqref{eq:optimal control2}  is equivalent to minimizing
\begin{subequations}
\label{eq:optimal control1}
\begin{align} \label{eq:optimal control1a}
\frac12\|\widetilde{F} u-w^\star\|_W^2+\frac{\varrho}{2}(\|\widetilde{z}\|^2_{Y'}+\|z_3\|_{Z_3}^2)
\end{align}
over
\begin{align}\label{eq:optimal control1b}
\set{(u,\widetilde{z},z_3) \in X \times Y'\times Z_3} {(B,\gamma_0)u=(f_1^\star+\divv_{\bf x} {\bf f}_2^\star+ \widetilde{z},u_0^\star+z_3)}.
\end{align}
\end{subequations} 
Indeed, from Theorem~\ref{thm:fosls} we know that for $\widetilde{z}=z_1+\divv_{\bf x} {\bf z}_2$, 
$(u,\widetilde{z},z_3)$ is in the set \eqref{eq:optimal control1b} if and only if $({\bf u},{\bf z})$ with ${\bf u}=(u,-{\bf A}\nabla_{\bf x} u-{\bf f}_2^\star-{\bf z}_2)$ is in the set \eqref{eq:optimal control1b}.
Knowing that for any $(z_1,{\bf z}_2) \in L_2(Q)\times L_2(Q)^d$, $\widetilde{z}:=z_1+\divv_{\bf x} {\bf z}_2 \in Y'$ and conversely, any $\widetilde{z} \in Y'$ can be written as $\widetilde{z}=z_1+\divv_{\bf x} {\bf z}_2 \in Y'$, where, as shown in Remark~\ref{rem:splitting}, for the pair $(z_1,{\bf z}_2) \in L_2(Q)\times L_2(Q)^d$ with smallest $\|z_1\|_{L_2(Q)}^2+\|{\bf z}_2\|_{L_2(Q)^d}^2$ it holds that $\|\widetilde{z}\|_{Y'}^2= \|z_1\|_{L_2(Q)}^2+\|{\bf z}_2\|_{L_2(Q)^d}^2$, the proof of equivalence of \eqref{eq:optimal control2} and \eqref{eq:optimal control1} is completed.

In particular, this shows that~\eqref{eq:optimal control2} covers the recently considered optimal control problem from~\cite{lsty21b}, where $\widetilde F={\rm Id}:X\to L_2(Q)$ and $Z_3=\{0\}$.
\end{remark}

\subsection{Formulation as saddle-point problem}\label{sec:saddle point}

Writing \eqref{eq:first-order system} with right-hand side ${\bf f}={\bf f}^\star+{\bf z}$ as $\langle G {\bf u},G {\bf v}\rangle_L=\langle {\bf f}^\star+{\bf z},G {\bf v}\rangle_L$ for all ${\bf v} \in U$, and
introducing the bilinear forms
$$
a({\bf u},{\bf v}) := \dual{G{\bf u}}{G{\bf v}}_L,\,\,
b({\bf z},{\bf v}) := -\dual{{\bf z}}{G{\bf v}}_L,\,\,
d({\bf u},{\bf v}) := \dual{F {\bf u}}{F{\bf v}}_W,\,\,
e({\bf z},{\bf y}) := \varrho \,\dual{{\bf z}}{{\bf y}}_Z,
$$
and the linear forms
$$
f^\star({\bf v}) := \dual{{\bf f}^\star}{G{\bf v}}_L,\quad
g^\star\big(({\bf u},{\bf z})\big) := \dual{F {\bf u}}{w^\star}_W, 
$$
for ${\bf u},{\bf v}\in U$ and ${\bf z}, {\bf y}\in Z$, and noting that the constant term $\norm{w^\star}{W}^2$ in $\norm{F{\bf u} - w^\star}{W}^2$ can be neglected for minimization,  the optimal control problem~\eqref{eq:optimal control2} can be rewritten as 
\begin{align}\label{eq:optimal control abstract}
  \operatorname*{argmin}\limits_{\{({\bf u},{\bf z})\in U\times Z:\, a({\bf u},{\bf v}) + b({\bf z},{\bf v}) = f^\star({\bf v})\, ({\bf v}\in U)\}} 
  \frac12 d({\bf u},{\bf u}) + \frac12 e({\bf z},{\bf z}) - g^\star\big(({\bf u},{\bf z})\big). 
\end{align}
The following lemma implies well-posedness and equivalence to a saddle-point problem.

\begin{lemma}\label{lem:optimal control}
Let $U$ and $Z$ be \emph{arbitrary} Hilbert spaces, $a:U\times U\to \R$, $b: Z\times U\to \R$, $d:U\times U\to\R$, $e:Z\times Z\to\R$ \emph{arbitrary} bounded bilinear forms such that  $a$ and $e$ are coercive, $d$ is positive semi-definite, and $f:U\to\R$, $g^\star:U\times Z\to\R$ are bounded linear forms. 
Then there exists a unique $({\bf u},{\bf z},{\bf p})\in U\times Z\times U$ that solves the saddle-point problem
\begin{equation}\label{eq:saddle}
\begin{array}{lcll}
 d({\bf u},{\bf v}) + e({\bf z},{\bf y}) + a({\bf v},{\bf p}) + b({\bf y},{\bf p}) & = & g^\star\big(({\bf v},{\bf y})\big) & \text{for all }({\bf v},{\bf y})\in U\times Z,
 \\
a({\bf u},{\bf q}) + b ({\bf z},{\bf q}) & = & f^\star({\bf q}) & \text{for all }{\bf q}\in U,
\end{array}
\end{equation}
and 
\begin{align}
 \norm{{\bf u}}{U} + \norm{{\bf z}}{Z}  + \norm{{\bf p}}{U} \le \const{stab} \big(\norm{f^\star}{U'} + \norm{g^\star}{U' \times Z'}  \big)
\end{align}
with a constant $\const{stab}>0$ that depends only on the continuity constants of $a,b,d,e$ and the coercivity constants of $a,e$.

Assuming symmetry of $d$ and $e$, the pair $({\bf u},{\bf z})\in U\times Z$ is the unique solution 
of \eqref{eq:optimal control abstract}. In this setting, ${\bf p}\in U$ is known as the \emph{co-state}.
\end{lemma}

\begin{proof}
For the first statement, we only have to show that the bilinear form $\big((\bf{u},{\bf z}),({\bf v},{\bf y})\big) \mapsto d({\bf u},{\bf v}) + e({\bf z},{\bf y})$ is coercive on the kernel of the operator $B:U\times Z\to U'$ defined by $(B({\bf u},{\bf z}))({\bf q}) := a({\bf u},{\bf q}) + b({\bf z},{\bf q})$, and that there holds the LBB (Ladyshenskaja--Babu\v{s}ka--Brezzi) condition 
\begin{align*}
 \sup_{({\bf v},{\bf y})\in U\times Z} \frac{a({\bf v},{\bf p}) + b({\bf y},{\bf p})}{\norm{{\bf v}}{U} + \norm{{\bf y}}{Z}} \gtrsim \norm{{\bf p}}{U} \quad\text{for all }{\bf p}\in U.
\end{align*}
Choosing ${\bf v} = {\bf p}$ as well as ${\bf y} = 0$, coercivity of $a$ gives the latter inequality. 
To see coercivity on the kernel, let $({\bf u},{\bf z})\in U\times Z$ with $a({\bf u},{\bf q}) + b({\bf z},{\bf q})=0$ for all ${\bf q}\in U$. 
Then, coercivity of $a$ and continuity of $b$ yield that 
\begin{align*}
 \norm{{\bf u}}{U}^2 \lesssim a({\bf u},{\bf u}) = |-b({\bf z},{\bf u})| \lesssim \norm{{\bf z}}{Z} \norm{{\bf u}}{U}, 
\end{align*}
or $\norm{{\bf u}}{U} \lesssim \norm{{\bf z}}{Z}$.
With positive semi-definiteness of $d$ and the coercivity of $e$, we conclude that 
\begin{align*}
 d({\bf u},{\bf u}) + e({\bf z},{\bf z}) \ge e({\bf z},{\bf z}) \gtrsim \norm{{\bf z}}{Z}^2 \gtrsim \norm{{\bf u}}{U}^2 + \norm{{\bf z}}{Z}^2
\end{align*}
and thus the proof of the first statement.

One easily verifies that \eqref{eq:optimal control abstract} has a unique solution $({\bf u},{\bf z})\in U\times Z$ that, assuming symmetric $d$ and $e$, solves
$a({\bf u},{\bf q}) + b ({\bf z},{\bf q})  =  f^\star({\bf q})$ for all ${\bf q}\in U$, and 
$d({\bf u},{\bf v}) + e({\bf z},{\bf y})  =  g^\star\big(({\bf v},{\bf y})\big)$ for all $({\bf v},{\bf y})\in U\times Z$ for which $a({\bf v},{\bf q}) + b({\bf y},{\bf q})=0$ for all ${\bf q} \in U$.
It is known that thanks to the LBB condition, these both equations uniquely determine the first two components of the solution $({\bf u},{\bf z},{\bf p})$ of \eqref{eq:saddle},
which proves the second statement.
%
\end{proof}

Clearly, Lemma~\ref{lem:optimal control} is also applicable for arbitrary closed subspaces $U^\delta \subset U$ and $Z^\delta \subset Z$ with the same uniform constant $\const{stab}>0$.
In particular, there exists a unique corresponding Galerkin solution $({\bf u}^\delta,{\bf z}^\delta,{\bf p}^\delta)\in U^\delta\times Z^\delta \times U^\delta$ of \eqref{eq:saddle}, which is even \emph{quasi-optimal}, i.e., 
\begin{align}\label{eq:quasi-optimal control}
\begin{split}
 &\norm{{\bf u}-{\bf u}^\delta}{U} + \norm{{\bf z} - {\bf z}^\delta}{Z} + \norm{{\bf p} - {\bf p}^\delta}{U} 
 \\
 &\qquad\le \const{opt} \inf_{({\bf v},{\bf y},{\bf q})\in U^\delta\times Z^\delta\times U^\delta} \big(\norm{{\bf u} - {\bf v}}{U} + \norm{{\bf z} - {\bf y}}{Z} + \norm{{\bf p} - {\bf q}}{U}\big),
\end{split}
\end{align} 
where $\const{opt}>0$ is proportional to the product of $\const{stab}$ and the maximum $\const{cont}$ of 
the continuity constants of $a$, $b$, $d$, and $e$ (see, e.g., \cite[Rem.~3.2]{sw21a}).
Fixing the parabolic PDE \eqref{eq:parabolic pde}, the spaces $Z$ and $W$, and the mapping $F$, from $e({\bf z},{\bf y}) = \varrho \,\dual{{\bf z}}{{\bf y}}_Z$ one infers that $\const{cont} \eqsim \max(\varrho,1)$, and 
$\const{stab} \eqsim \min(\frac{1}{\varrho},1)$ (see, e.g., \cite[Thm.~49.13]{eg21_2}), and thus
\begin{align*}
\const{opt} \eqsim \max(\tfrac{1}{\varrho},\varrho).
\end{align*}

\subsection{Optimality system of PDEs} \label{sec:optimality system}
Let $\widehat{L}:=\clos_L Z$, so that $Z\hookrightarrow \widehat{L} \simeq \widehat{L}' \hookrightarrow Z'$ is a Gelfand triple, let $\Pi \in \cL(L,L)$ be the orthogonal projector onto $\widehat{L}$, $\bm{\ell}:=G{\bf p}$,
and let $C \in \LL(Z,Z')$ be such that $\langle \cdot,\cdot\rangle_Z=\langle C\cdot,\cdot\rangle_{\widehat{L}}$.

Then the first equation in \eqref{eq:saddle} reads as 
$$
\varrho \langle C {\bf z},{\bf y}\rangle_{\widehat{L}}-\langle \Pi \bm{\ell},{\bf y}\rangle_{\widehat{L}}+
\langle G {\bf v},\bm{\ell}\rangle_L=\langle F {\bf v},w^\star-F {\bf u}\rangle_W \text{ for all } ({\bf v},{\bf y}) \in U \times Z.
$$
i.e., ${\bf z}=\frac{1}{\varrho} C^{-1} \Pi \bm{\ell}$ and $\bm{\ell}=G^{-*}(F^*(w^\star-F {\bf u}))$ with $G^*$ and $F^*$ the Hilbert adjoints of $G\in \cL(U,L)$ and $F\in \cL(U,W)$, respectively.
Together with ${\bf u}=G^{-1}({\bf f}^\star +\frac{1}{\varrho} C^{-1} \Pi \bm{\ell})$, this last equation forms the coupled optimality system associated to our optimal control problem.

To derive a first-order PDE for $\bm{\ell}$, we consider the case that
\begin{equation} \label{eq:F}
F{\bf v}=(\chi_1 v_1,\chi_2 {\bf v}_2,\chi_3v_1(T,\cdot)),\,W=L_2(Q)\times L_2(Q)^d \times L_2(\Omega),\,w^\star=(w_1^\star,{\bf w}_2^\star,w_3^\star),
\end{equation}
for some $\chi_1,\chi_2 \in L_\infty(Q)$, $\chi_3 \in L_\infty(\Omega)$, possibly with one or more $\chi_i$ being zero.
We further make the additional regularity assumption $\divv_{\x} {\bf b}\in L_\infty(Q)$.
With the outer normal vector ${\bf n}_\x$ on $\partial\Omega$, (formal)
 integration-by-parts shows that
 \begin{equation} \label{eq:int-by-parts}
 \begin{split}
 \langle G {\bf v},\bm{\ell}\rangle_L= &
  \dual{-\partial_t \ell_{1}+\divv_{\bf x} {\bf A} \bm{\ell}_2- {\bf b}\cdot \nabla_{\bf x} \ell_1+(c- \divv_{\bf x}{\bf b})\ell_1}{v_1}_{L_2(Q)} \\
  &- \dual{\nabla_\x \ell_{1}+\bm{\ell}_{2}}{{\bf v}_2}_{L_2(Q)^d}  + \dual{\ell_{1}(t,\cdot)}{v_1(t,\cdot)}_{L_2(\Omega)}\big|_{t=0}^{t=T} \\
& + \int_0^T\dual{\ell_{1}(t,\cdot) {\bf v}_2(t,\cdot)}{{\bf n}_\x}_{L_2(\partial\Omega)}\d t+\dual{\ell_{3}}{v_1(0,\cdot)}_{L_2(\Omega)}.
\end{split}
\end{equation}
From $\langle G {\bf v},\bm{\ell}\rangle_L=\langle F {\bf v},w^\star-F {\bf u}\rangle_W$ for all ${\bf v} \in U$, we infer that
$\ell_3 =\ell_{1}(0,\cdot)$, and
\begin{equation}\label{eq:adjoint pde}
\begin{array}{rcll}
-\partial_t \ell_{1}+\divv_{\bf x} {\bf A} \bm{\ell}_2- {\bf b}\cdot \nabla_{\bf x} \ell_1+(c- \divv_{\bf x}{\bf b})\ell_1 & = & \chi_1(w_1^\star - \chi_1 u_1) & \text{ in } Q,\\
-\nabla_\x \ell_{1}-\bm{\ell}_{2}& = & \chi_2({\bf w}_2^\star - \chi_2 {\bf u}_2) & \text{ in } Q,\\
 \ell_1 & = & 0 & \text{ on }\Sigma,\\
 \ell_1(T,\cdot) & = & \chi_3(w_3^\star - \chi_3 u_1(T,\cdot)) & \text{ on }\Omega,
\end{array}
\end{equation}
i.e., $\ell_1$ is the solution of a backward parabolic problem.

\subsection{A posteriori error estimation} \label{sec:a posteriori}
The well-posedness of \eqref{eq:saddle} shows that for $({\bf u},{\bf z},{\bf p})$ being its solution, and  any $({\bf u}^\delta,{\bf z}^\delta,{\bf p}^\delta) \in U \times Z \times U$,
\begin{align} \nonumber 
&\|{\bf u}-{\bf u}^\delta\|^2_U+\|{\bf z}-{\bf z}^\delta\|^2_Z+\|{\bf p}-{\bf p}^\delta\|^2_U \eqsim\\  \nonumber
&\hspace*{-5em}\sup_{\mbox{}\hspace*{5em}0 \neq ({\bf v},{\bf y},{\bf q}) \in U \times Z \times U}\hspace*{-4em} \frac{\big[d({\bf u}\!-\!{\bf u}^\delta,{\bf v})\!+\!e({\bf z}\!-\!{\bf z}^\delta,{\bf y})\!+\!a({\bf v},{\bf p}\!-\!{\bf p}^\delta)\!+\!b({\bf y},{\bf p}\!-\!{\bf p}^\delta)\!+\!a({\bf u}\!-\!{\bf u}^\delta,{\bf q})\!+\!b({\bf z}\!-\!{\bf z}^\delta,{\bf q})\big]^2}{\|{\bf v}\|^2_U\!+\!\|{\bf y}\|^2_Z\!+\!\|{\bf q}\|^2_U}=\\  \nonumber
&\hspace*{-5em}\sup_{\mbox{}\hspace*{5em}0 \neq ({\bf v},{\bf y},{\bf q}) \in U \times Z \times U}\hspace*{-4em}  \frac{\big[g^\star(({\bf v},{\bf y}))\!+\!f({\bf q})\!-\!\Big(d({\bf u}^\delta,{\bf v})\!+\!e({\bf z}^\delta,{\bf y})\!+\!a({\bf v},{\bf p}^\delta)\!+\!b({\bf y},{\bf p}^\delta)\!+\!a({\bf u}^\delta,{\bf q})\!+\!b({\bf z}^\delta,{\bf q})\Big)\big]^2}{\|{\bf v}\|^2_U\!+\!\|{\bf y}\|^2_Z\!+\!\|{\bf q}\|^2_U}
=\\ \nonumber
&\hspace*{-5em}\sup_{\mbox{}\hspace*{5em}0 \neq ({\bf v},{\bf y},{\bf q}) \in U \times Z \times U}\hspace*{-4em}
\frac{\big[\langle F {\bf v}, w^\star\!-\!F {\bf u}^\delta \rangle_W
\!-\!\langle G {\bf v},G {\bf p}^\delta \rangle_L\!+\!
\langle {\bf f}^\star\!+\!{\bf z}^\delta\!-\!G {\bf u}^\delta,G{\bf q}\rangle_L
\!+\!
\langle \Pi G {\bf p}^\delta\!-\!\varrho C {\bf z}^\delta,{\bf y}\rangle_{\widehat{L}}\big]^2}{\|{\bf v}\|^2_U\!+\!\|{\bf y}\|^2_Z\!+\!\|{\bf q}\|^2_U}=\\ \nonumber
& \sup_{0 \neq {\bf v} \in U}
\frac{\big[\langle F {\bf v}, w^\star\!-\!F {\bf u}^\delta \rangle_W
\!-\!\langle G {\bf v},G {\bf p}^\delta \rangle_L\big]^2}{\|{\bf v}\|^2_U}
\!+\!
\sup_{0 \neq {\bf q} \in U}\frac{\langle {\bf f}^\star\!+\!{\bf z}^\delta\!-\!G {\bf u}^\delta,G{\bf q}\rangle_L^2}{\|{\bf q}\|_U^2}
\!+\!
\|\Pi  G {\bf p}^\delta\!-\!\varrho C {\bf z}^\delta\|^2_{Z'} \eqsim\\ \label{eq:3terms}
&
 \sup_{0 \neq {\bf v} \in U}
\frac{\big[\langle F {\bf v}, w^\star\!-\!F {\bf u}^\delta \rangle_W
\!-\!\langle G {\bf v},G {\bf p}^\delta \rangle_L\big]^2}{\|{\bf v}\|^2_U}\!+\!
\| {\bf f}^\star\!+\!{\bf z}^\delta\!-\!G {\bf u}^\delta\|_L^2
\!+\!
\|\Pi  G {\bf p}^\delta\!-\!\varrho C {\bf z}^\delta\|^2_{Z'}
\end{align}
from $G\colon U \rightarrow L$ being a linear isomorphism. The hidden constants absorbed by the two $\eqsim$-symbols can be quantified in terms of the well-posedness of \eqref{eq:saddle} and that of $G$.

The second term in \eqref{eq:3terms} is computable, and in any case when the topology of $Z$ equals that of $\widehat{L}$ 
and the application of $\Pi$ is computable, so is the third. 
To estimate the first term, we briefly discuss two possibilities.

\begin{remark}
For the case that  $Z=L_2(Q)\times\{0\}\times\{0\}$ and $(\chi_1,\chi_2,\chi_3)=(\ind,0,0)$, a functional estimator was introduced in \cite{ls21}.
For arbitrary approximations of the state $u\in X$ and the co-state $p\in X$, this estimator is a computable guaranteed upper bound for the error in the $X$-norm.
\end{remark}


\subsubsection{A reliable estimator}\label{sec:reliable} We consider the case that $F$ is of the form given in \eqref{eq:F} and abbreviate $\bm{\ell}^\delta=(\ell_1^\delta,\bm{\ell}_2^\delta,\ell_3^\delta):=G {\bf p}^\delta$.
In view of \eqref{eq:int-by-parts}--\eqref{eq:adjoint pde}, any (sufficiently smooth) $\widetilde{\bm{\ell}}^\delta=(\widetilde{\ell}_1^\delta,\widetilde{\bm{\ell}}_2^\delta,\widetilde{\ell}_3^\delta)$ with
$(\widetilde{\ell}_1^\delta)|_{\Sigma}=0$ and $\widetilde{\ell}_3^\delta=\widetilde{\ell}_1^\delta(0,\cdot)$ provides the upper bound
\begin{align*}
\sup_{0 \neq {\bf v} \in U}
&\frac{\langle F {\bf v}, w^\star-F {\bf u}^\delta \rangle_W
-\langle G {\bf v},G {\bf p}^\delta \rangle_L}{\|{\bf v}\|_U} \leq \|G\|_{\cL(U,L)} \|\widetilde{\bm{\ell}}^\delta-\bm{\ell}^\delta\|_L+
\\
&\|-\partial_t \widetilde{\ell}_1^\delta+\divv_{\bf x} {\bf A} \widetilde{\bm{\ell}}_2^\delta- {\bf b}\cdot \nabla_{\bf x} \widetilde{\ell}_1^\delta+(c- \divv_{\bf x}{\bf b})\widetilde{\ell}_1^\delta - \chi_1(w_1^\star - \chi_1 u_1^{\delta}) \|_{L_2(Q)}+\\
&
\|-\nabla_\x \widetilde{\ell}_1^\delta-\widetilde{\bm{\ell}}_2^\delta- \chi_2({\bf w}_2^\star - \chi_2 {\bf u}_2^{\delta})\|_{L_2(Q)^d}+\\
&\|{\bf v}\mapsto v_1(T,\cdot)\|_{\cL(U,L_2(\Omega))}
\| \widetilde{\ell}_1^\delta(T,\cdot) - \chi_3(w_3^\star - \chi_3 u_1^{\delta}(T,\cdot)) \|_{L_2(\Omega)}.
\end{align*}
Choosing $(\widetilde\ell_1^\delta,\widetilde{\bm\ell}_2^\delta)\in U$ as the solution of \eqref{eq:adjoint pde} with ${\bf u}$ replaced by the computable approximation ${\bf u}^\delta$, only the term $\|G\|_{\cL(U,L)} \|\widetilde{\bm{\ell}}^\delta-\bm{\ell}^\delta\|_L$ would be present in the upper bound, as then $\langle F {\bf v}, w^\star-F {\bf u}^\delta \rangle_W=\langle G{\bf v},\widetilde{\bm\ell}^\delta\rangle_L$.
In this case, the term $\|\widetilde{\bm{\ell}}^\delta-\bm{\ell}^\delta\|_L$ would even be a reliable and efficient estimator for the first term in~\eqref{eq:3terms}.
However, in general the mentioned solution is not computable so that one has to approximate it, e.g., by a Galerkin method. 
Inspired by \cite{ls21}, a computationally cheaper option would be to simply post-process $\bm{\ell}^\delta$ to obtain $\widetilde{\bm{\ell}}^\delta$, i.e., apply a suitable smoothening quasi-interpolator to the in general non-smooth $\bm{\ell}^\delta$.

\subsubsection{An efficient estimator}\label{sec:efficient}
 A computable efficient estimator is obtained by replacing the supremum over ${\bf v} \in U$ in the first term in \eqref{eq:3terms}
by a supremum over a finite-dimensional subspace $U^{\widetilde \delta} \subset U$.
Since for $({\bf u}^\delta,{\bf z}^\delta,{\bf p}^\delta)$ being the Galerkin solution of \eqref{eq:saddle} from $U^\delta \times Z^\delta \times U^\delta$, the numerator in the first term vanishes for ${\bf v} \in U^\delta$, $U^{\widetilde \delta}$ needs to be a `sufficient' enlargement of 
$U^\delta$.

The resulting estimator can be proven to be even equivalent to the error whenever there exists a projector $P^{\delta} \in \cL(U,U)$, bounded uniformly in $\delta$, with
$\ran P^{\delta} \subset U^{\widetilde \delta}$ and  $
\langle F ({\rm Id} - P^{\delta}){\bf v}, w^\star-F {\bf v}^\delta \rangle_W
-\langle G ({\rm Id} - P^{\delta}){\bf v},G {\bf q}^\delta \rangle_L=0$ for all ${\bf v} \in U$ and ${\bf v}^\delta,{\bf q}^\delta  \in U^\delta$.
For $U^\delta$ being a finite element space with respect to~a general partition of $Q$, i.e., not being a product of partitions of $I$ and $\Omega$, the construction of such `Fortin' projectors however seems hard.

\subsection{Numerical experiments}\label{sec:numerics optimal control}
We consider the heat equation, i.e., ${\bf A} := {\bf Id}$, ${\bf b} := 0$, and $c:=0$ in \eqref{eq:parabolic pde}.
Moreover, we set ${\bf f}^\star:=(f_1^\star,0,u_0^\star)$, $W:=L_2(Q)$, $F\colon{\bf v} \mapsto v_1$,
$Z:=L_2(Q) \times \{0\} \times \{0\} \simeq L_2(Q)$.
In particular, with $u=u_1$ the optimal control problem \eqref{eq:optimal control2} in strong form reads as 
\begin{align}\label{eq:optimal control example}
 \operatorname*{argmin}\limits_{\{(u,z)\in X\times L_2(Q):\,\partial_t u -\Delta_\x u = f_1^\star + z \wedge u(0,\cdot) = u^\star_0\}} 
 \frac12\norm{u - w^\star}{L_2(Q)}^2 + \frac{\varrho}{2} \norm{z}{L_2(Q)}^2. 
\end{align}

In this case, the optimality system derived in Section~\ref{sec:optimality system} reads as
$$
\begin{array}{rcll}
 \partial_t u_1 - \Delta_\x  u_1 & = & f_1^\star+\frac{1}{\varrho} \ell_1& \text{ on } Q,\\
 u_1 & = & 0 & \text{ on }\Sigma,\\
u_1(0,\cdot) & = & u^\star_0 & \text{ on }\Omega,
\end{array}
$$
and
$$
\begin{array}{rcll}
- \partial_t \ell_1 - \Delta_\x  \ell_1 & = & w^\star-u_1& \text{ on } Q,\\
 \ell_1 & = & 0 & \text{ on }\Sigma,\\
\ell_1(T,\cdot) & = & 0 & \text{ on }\Omega,
\end{array}
$$
where ${\bf u}_2=-\nabla_{\bf x} u_1$, $\bm{\ell}_2=-\nabla_{\bf x} \ell_1$, and $\ell_3=\ell_1(0,\cdot)$.
By prescribing arbitrary $u_1 \in X$, $\ell_1 \in L_2(Q)$ with $\partial_t u_1 - \Delta_\x  u_1 \in L_2(Q)$,
$- \partial_t \ell_1 - \Delta_\x  \ell_1  \in L_2(Q)$, $\ell_1(0,\cdot) \in L_2(\Omega)$,
$\ell_1=0$  on $\Sigma$, and $\ell_1(T,\cdot)=0$, the parameters
$f_1^\star$, $u^\star_0$, and $w^\star$ are determined by the optimality system.

The control $z$ and co-state ${\bf p}$ are determined by $z=\frac{1}{\varrho} \ell_1$ and $G {\bf p}=\bm{\ell}$, i.e.,
$$
\begin{array}{rcll}
\partial_t p_1 - \Delta_\x p_1 & = & \ell_1- \Delta_\x \ell_1& \text{ on } Q,\\
 p_1 & = & 0 & \text{ on }\Sigma,\\
p_1(0,\cdot) & = & \ell_1(0,\cdot) & \text{ on }\Omega,
\end{array}
$$
and ${\bf p}_2=\nabla_\x (\ell_1-p_1)$.

Given some conforming quasi-uniform partition $\TT^\delta$ of the space-time cylinder $Q$ into $(d+1)$-simplices, we take 
$U^\delta$ to be the $(d+1)$-fold Cartesian product of continuous piecewise affine functions 
with the first-coordinate space restricted by homogeneous Dirichlet boundary conditions on $\Sigma$, and $Z^\delta$ being the space of piecewise constants.

Taking sufficiently smooth $u_1$ and $\ell_1$, in view of \eqref{eq:quasi-optimal control} for the Galerkin solution $({\bf u}^\delta,z^\delta,{\bf p}^\delta) \in U^\delta \times Z^\delta \times U^\delta$, we obtain $\norm{{\bf u}-{\bf u}^\delta}{U} + \norm{z - z^\delta}{L_2(Q)} + \norm{{\bf p} - {\bf p}^\delta}{U} = \OO({\rm dofs}^{-\frac{1}{d+1}})$ assuming that
 $\inf_{{\bf q} \in U^\delta}\|{\bf p}-{\bf q}\|_U = \OO({\rm dofs}^{-\frac{1}{d+1}})$.

Concerning the latter, assuming the compatibility conditions $\ell_1(0,\cdot) \in H^1_0(\Omega)$,
$(\ell_1- \Delta_\x \ell_1)(0,\cdot)+\Delta_\x \ell_1(0,\cdot) \in H^1_0(\Omega)$,
and
$\partial_t(\ell_1- \Delta_\x \ell_1)(0,\cdot)+\Delta_\x(\ell_1- \Delta_\x \ell_1)(0,\cdot)+\Delta_\x^2\ell_1(0,\cdot) \in L_2(\Omega)$, \cite[Theorem~27.2]{wloka87} shows that
$p_1 \in H^2(I;H^1_0(\Omega))$.
For a smooth or convex $\Omega$, by interchanging $\partial_t$ and $\Delta_\x$, from $\partial_t(\ell_1- \Delta_\x \ell_1-\partial_t p_1) \in L_2(I;L_2(\Omega))$
regularity of the Poisson problem shows that $\partial_t p_1 \in L_2(I;H^2(\Omega))$, i.e., $p_1 \in H^1(I;H^2(\Omega))$. 
For a smooth $\Omega$,  from $\ell_1- \Delta_\x \ell_1-\partial_t p_1 \in L(I;H^1(\Omega))$, regularity of the Poisson problem shows that $p_1 \in L_2(I;H^3(\Omega))$.
For $\Omega$ being a square, the latter should be read as $p_1 \in L_2(I;H^{3-\varepsilon}(\Omega))$ for any $\varepsilon>0$ (\cite{kondratiev70}, cf.~also \cite[Ex.~9.1.25]{hackbusch92}).
Pretending that $\varepsilon=0$, we conclude that $p_1 \in H^2(Q)$ and ${\bf p}_2 \in H^2(Q)^d$, so that
$\inf_{{\bf q} \in U^\delta}\|{\bf p}-{\bf q}\|_U \lesssim \inf_{{\bf q} \in U^\delta}\|{\bf p}-{\bf q}\|_{H^1(Q) \times H^1(Q)^d}= \OO({\rm dofs}^{-\frac{1}{d+1}})$,
which is thus only `nearly' demonstrated for $\Omega$ being a square.

\subsubsection{Experiment in 1+1D}\label{sec:control 1d}
Let $\Omega:=(0,1)$, $T:=1$, and $\varrho=0.01$. 
We prescribe 
\begin{align*}
u(t,x) =u_1(t,x):=\cos(\pi t) \sin(\pi x), \quad
 \ell_1(t,x):=\varrho (1-t) \sin(\pi x)
\end{align*} 
on the space-time cylinder $Q=(0,1)^2$, and determine $f_1^\star$, $u_0^{\star}$, and $w^\star$ by the optimality system. 

Starting on an initial triangulation $\TT^\delta$ of $Q$ with two elements, we define a sequence of uniform triangulations $\TT^\delta$ by splitting  all elements in the previous $\TT^\delta$ into four new elements by repeated newest vertex bisection.
The convergence plot for the resulting Galerkin approximations $({\bf u}^\delta, z^\delta,{\bf p}^\delta)\in U^\delta\times Z^\delta\times U^\delta$ of \eqref{eq:saddle} is displayed in Figure~\ref{fig:control 1d}. 
While, as expected, $\norm{\nabla_\x (u-u^\delta)}{L_2(Q)}$ and $\norm{z-z^\delta}{L_2(Q)}$ converge at rate $\OO({\rm dofs}^{-\frac12})$, 
$\norm{u(0,\cdot) -u^\delta(0,\cdot)}{L_2(\Omega)}$, $\norm{u(T,\cdot) -u^\delta(T,\cdot)}{L_2(\Omega)}$, $\norm{u -u^\delta}{L_2(Q)}$, and $|J({\bf u},z) - J({\bf u}^\delta,z^\delta)|$ even converge with the double rate. 

\begin{center}
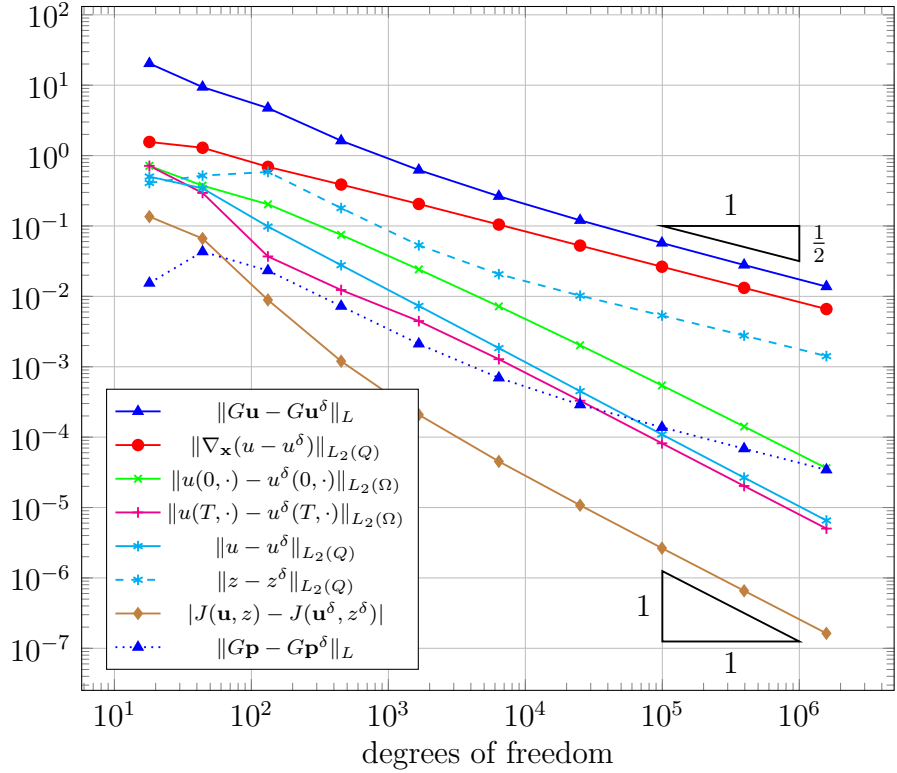
\begin{figure}
\begin{tikzpicture}
\begin{loglogaxis}[
width = 0.75\textwidth,
clip = true,
xlabel= {degrees of freedom},
grid = major,
legend pos = south west,
legend style={font=\tiny},
legend entries = {$\norm{G{\bf u} - G{\bf u}^\delta}{L}$\\ $\norm{\nabla_\x (u-u^\delta)}{L_2(Q)}$\\ $\norm{u(0,\cdot) -u^\delta(0,\cdot)}{L_2(\Omega)}$ \\ $\norm{u(T,\cdot) -u^\delta(T,\cdot)}{L_2(\Omega)}$ \\ $\norm{u -u^\delta}{L_2(Q)}$ \\ $\norm{z-z^\delta}{L_2(Q)}$\\ $|J({\bf u},z) - J({\bf u}^\delta,z^\delta)|$\\ $\norm{G{\bf p} - G{\bf p}^\delta}{L}$\\} 
]
\addplot[line width = 0.25mm,color=blue,mark=triangle*] table[x=ndof,y=err_U, col sep=comma] {data/stcont_D1_order1.csv};
\addplot[line width = 0.25mm,color=red,mark=*] table[x=ndof,y=err_Y, col sep=comma] {data/stcont_D1_order1.csv};
\addplot[line width = 0.25mm,color=green,mark=x] table[x=ndof,y=err_0, col sep=comma] {data/stcont_D1_order1.csv};
\addplot[line width = 0.25mm,color=magenta,mark=+] table[x=ndof,y=err_T, col sep=comma] {data/stcont_D1_order1.csv};
\addplot[line width = 0.25mm,color=cyan,mark=asterisk] table[x=ndof,y=err_l2, col sep=comma] {data/stcont_D1_order1.csv};
\addplot[dashed, line width = 0.25mm,color=cyan,mark=asterisk, mark options=solid] table[x=ndof,y=err_z1, col sep=comma] {data/stcont_D1_order1.csv};
\addplot[line width = 0.25mm,color=brown,mark=diamond*] table[x=ndof,y=err_J, col sep=comma] {data/stcont_D1_order1.csv};
\addplot[dotted,line width = 0.25mm,color=blue,mark=triangle*,mark options=solid] table[x=ndof,y=err_p, col sep=comma] {data/stcont_D1_order1.csv};

\tikzmath{\xc=100000;\yc=0.1;\slope=1/2;}
\draw[line width = 0.25mm] (axis cs:\xc,\yc) -- (axis cs:10*\xc,\yc) -- (axis cs:10*\xc,\yc * 0.1^\slope) -- cycle;
\node[above] at (axis cs:\xc*10^0.5,\yc) {$1$};
\node[right] at (axis cs:10*\xc,\yc * 0.1^0.5^\slope) {$\frac12$};

\tikzmath{\xcc=100000;\ycc=0.00000125;\slopee=1;}
\draw[line width = 0.25mm] (axis cs:\xcc,\ycc) -- (axis cs:\xcc,\ycc * 0.1^\slopee) -- (axis cs:10*\xcc,\ycc * 0.1^\slopee) -- cycle;
\node[left] at (axis cs:\xcc,\ycc * 0.1^0.5^\slopee) {$1$};
\node[below] at (axis cs:\xcc*10^0.5,\ycc * 0.1^\slopee) {$1$};

\end{loglogaxis}
\end{tikzpicture}
\caption{\label{fig:control 1d}
Optimal control problem in 1+1D of Section~\ref{sec:control 1d}.}
\end{figure}
\end{center}

\subsubsection{Experiment in 2+1D}\label{sec:control 2d} 
Let $\Omega:=(0,1)^2$, $T:=1$, and $\varrho = 0.01$. 
We prescribe 
\begin{align*}
 u(t,x_1,x_2):=\cos(\pi t) \sin(\pi x_1) \sin(\pi x_2), \quad  \ell_1(t,x_1,x_2):=\varrho (1-t) \sin(\pi x_1)\sin(\pi x_2)
\end{align*}
on the space-time cylinder $Q=(0,1)^3$, and we choose $f_1^\star$, $u_0^\star$, and $w^\star$ by the optimality system.  


Starting on an initial triangulation $\TT^\delta$ of $Q$ with $12$ elements, we define a sequence of quasi-uniform triangulations $\TT^\delta$ by  splitting all elements in the previous $\TT^\delta$ into eight new elements by repeated newest vertex bisection.
The convergence plot for the resulting Galerkin approximations $({\bf u}^\delta, z^\delta,{\bf p}^\delta)\in U^\delta\times Z^\delta\times U^\delta$ of \eqref{eq:saddle} is displayed in Figure~\ref{fig:control 2d}. 
While, as expected, $\norm{\nabla_\x (u-u^\delta)}{L_2(Q)}$ and $\norm{z-z^\delta}{L_2(Q)}$ converge at rate $\OO({\rm dofs}^{-\frac13})$, 
$\norm{u(0,\cdot) -u^\delta(0,\cdot)}{L_2(\Omega)}$, $\norm{u(T,\cdot) -u^\delta(T,\cdot)}{L_2(\Omega)}$, $\norm{u -u^\delta}{L_2(Q)}$, and $|J({\bf u},z) - J({\bf u}^\delta,z^\delta)|$ even converge (at least almost) with the double rate.

\begin{center}
\begin{figure}
\begin{tikzpicture}
\begin{loglogaxis}[
width = 0.75\textwidth,
clip = true,
xlabel= {degrees of freedom},
ymin=0.0000015,
grid = major,
legend pos = south west,
legend style={font=\tiny},
legend entries = {$\norm{G{\bf u} - G{\bf u}^\delta}{L}$\\ $\norm{\nabla_\x (u-u^\delta)}{L_2(Q)}$\\ $\norm{u(0,\cdot) -u^\delta(0,\cdot)}{L_2(\Omega)}$ \\ $\norm{u(T,\cdot) -u^\delta(T,\cdot)}{L_2(\Omega)}$ \\ $\norm{u -u^\delta}{L_2(Q)}$ \\ $\norm{z-z^\delta}{L_2(Q)}$\\ $|J({\bf u},z) - J({\bf u}^\delta,z^\delta)|$\\ $\norm{G{\bf p} - G{\bf p}^\delta}{L}$\\} 
]
\addplot[line width = 0.25mm,color=blue,mark=triangle*] table[x=ndof,y=err_U, col sep=comma] {data/stcont_D2_order1.csv};
\addplot[line width = 0.25mm,color=red,mark=*] table[x=ndof,y=err_Y, col sep=comma] {data/stcont_D2_order1.csv};
\addplot[line width = 0.25mm,color=green,mark=x] table[x=ndof,y=err_0, col sep=comma] {data/stcont_D2_order1.csv};
\addplot[line width = 0.25mm,color=magenta,mark=+] table[x=ndof,y=err_T, col sep=comma] {data/stcont_D2_order1.csv};
\addplot[line width = 0.25mm,color=cyan,mark=asterisk] table[x=ndof,y=err_l2, col sep=comma] {data/stcont_D2_order1.csv};
\addplot[dashed, line width = 0.25mm,color=cyan,mark=asterisk, mark options=solid] table[x=ndof,y=err_z1, col sep=comma] {data/stcont_D2_order1.csv};
\addplot[line width = 0.25mm,color=brown,mark=diamond*] table[x=ndof,y=err_J, col sep=comma] {data/stcont_D2_order1.csv};
\addplot[dotted,line width = 0.25mm,color=blue,mark=triangle*,mark options=solid] table[x=ndof,y=err_p, col sep=comma] {data/stcont_D2_order1.csv};

\tikzmath{\xc=100000;\yc=1;\slope=1/3;}
\draw[line width = 0.25mm] (axis cs:\xc,\yc) -- (axis cs:10*\xc,\yc) -- (axis cs:10*\xc,\yc * 0.1^\slope) -- cycle;
\node[above] at (axis cs:\xc*10^0.5,\yc) {$1$};
\node[right] at (axis cs:10*\xc,\yc * 0.1^0.5^\slope) {$\frac13$};

\tikzmath{\xcc=80000;\ycc=0.000125;\slopee=2/3;}
\draw[line width = 0.25mm] (axis cs:\xcc,\ycc) -- (axis cs:\xcc,\ycc * 0.1^\slopee) -- (axis cs:10*\xcc,\ycc * 0.1^\slopee) -- cycle;
\node[left] at (axis cs:\xcc,\ycc * 0.1^0.5^\slopee) {$\frac23$};
\node[below] at (axis cs:\xcc*10^0.5,\ycc * 0.1^\slopee) {$1$};

\end{loglogaxis}
\end{tikzpicture}
\caption{\label{fig:control 2d}
Optimal control problem in 2+1D of Section~\ref{sec:control 2d}.}
\end{figure}
\end{center}

\section{Time-dependent domains} \label{sec: time-dependent domains}
In this section, we assume that the spatial domain $\Omega$ changes throughout time. 
More precisely, let $\widehat\Omega\subset\R^d$ be a Lipschitz domain with boundary $\partial\widehat\Omega=\widehat\Gamma$, 
and $\widehat Q:=I\times\widehat\Omega$ the corresponding space-time cylinder with lateral boundary $\widehat\Sigma:=I\times\widehat\Gamma$. 
We denote the spaces $X$, $Y$, $U$, and $L$ on $\widehat{Q}$ by $\widehat X$, $\widehat{Y}$, $\widehat{U}$, and $\widehat{L}$, respectively.
We suppose that the actual space-time domain $Q\subset \R^{d+1}$ is given via a bijection of the form 
\begin{align*}
 \kappa: \overline{\widehat Q} \to \overline Q,\quad \left(\begin{array}{@{}c@{}}  t \\ \widehat\x \end{array} \right)
\mapsto
\left(\begin{array}{@{}c@{}}  t \\ \kappa'(t,\widehat x) \end{array} \right),
\end{align*}
where, for all $t\in\overline I$,  $\kappa'(t,\cdot): \widehat\Omega \to \R^d$ maps $\widehat\Omega$ bijectively onto some Lipschitz domain $\Omega_t\subset\R^d$.
We require the regularities  
$\kappa' \in W^1_\infty(\widehat{Q})^d$, 
$\essinf_{\widehat{Q}} \det {\rm D}_{\widehat\x} \kappa'>0$, $\sup_{t \in I} \|\kappa'(t,\cdot)\|_{W_\infty^3(\widehat{\Omega})^d}<\infty$, and $\sup_{t \in I} \|\partial_t \kappa'(t,\cdot)\|_{W_\infty^1(\widehat{\Omega})^d}<\infty$.
Note that 
\begin{align}\label{eq:Dkappa}
 {\rm D} \kappa = \left(\begin{array}{@{}cc@{}} 1 & 0\\ \partial_t\kappa' &{\rm D}_{\widehat\x}\kappa' \end{array} \right) 
 \quad \text{and}\quad \det {\rm D}\kappa = \det {\rm D}_{\widehat\x}\kappa'.
\end{align}

With the lateral boundary $\Sigma:=\kappa(\widehat\Sigma)$ and $\Omega:=\Omega_0$, we consider the parabolic PDE~\eqref{eq:parabolic pde} 
with ${\bf A}={\bf A}^\top\in L_\infty(Q)^{d\times d}$ uniformly positive, ${\bf b}\in L_\infty(Q)^d$, $c\in L_\infty(Q)$, 
$f_1\in L_2(Q)$, ${\bf f}_2\in L_2(Q)^d$, and $u_0\in L_2(\Omega_0)$.

\subsection{Formulation as first-order system}
As in the time-independent domain case, in first-order system formulation \eqref{eq:parabolic pde} reads as 
\begin{align}\label{eq:first-order system2}
 G{\bf u}:=  \left(\begin{array}{@{}c@{}} \divv {\bf u} +{\bf b}\cdot \nabla_\x u_1+ cu_1 \\ -{\bf u}_2 - {\bf A}\nabla_\x u_1 \\  u_1(0,\cdot)\end{array} \right) 
  =  \left(\begin{array}{@{}c@{}}  f_1 \\ {\bf f}_2\\ u_0 \end{array} \right)=:{\bf f},
\end{align}
which is again well-posed:

\begin{theorem} \label{thm:100} With $U$ and $L$ defined as in the time-independent domain case\footnote{Setting $Y:=\set{\widehat v \circ\kappa^{-1}}{\widehat v  \in \widehat{Y}}=\set{v \in L_2(Q)}{\nabla_\x v \in L_2(Q)^d\wedge v|_\Sigma = 0}$ equipped with the norm $\sqrt{\norm{v}{L_2(Q)}^2 + \norm{\nabla_{\bf x} v}{L_2(Q)^d}^2}$.}, $G$ is a linear isomorphism from $U$ to $L$.
\end{theorem}

\begin{proof}
We show that $G=F\circ \widetilde G\circ H$ with linear isomorphisms $F\in\cL(\widehat L,L), \widetilde G\in\cL(\widehat U,\widehat L)$, and $H\in\cL(U,\widehat U)$.

We set $H{\bf u}:=\widehat{\bf u}:=\det {\rm D}\kappa\, ({\rm D}\kappa)^{-1} {\bf u} \circ \kappa$. 
A familiar property of the Piola transformation (e.g. \cite[Lemma~9.6]{eg21_1}) is that
\be \label{101}
\divv \widehat{\bf u}=\det {\rm D}\kappa\, \divv {\bf u} \circ \kappa.
\ee
By definition of $\widehat{\bf u}$ we have
\be \label{102}
 u_1 \circ \kappa=(\det {\rm D}\kappa)^{-1} \widehat u_1, \quad {\bf u}_2 \circ \kappa=(\det {\rm D}\kappa)^{-1}[{\rm D}_{\widehat{\bf x}} \kappa' \,\widehat{\bf u}_2+\widehat{u}_1 \partial_t \kappa'].
\ee
Applications of the product and chain rules show that
\be\label{103}
\nabla_{{\bf x}} u_1 \circ \kappa=(\det {\rm D}\kappa)^{-1} ({\rm D}_{\widehat{\bf x}}\kappa')^{-\top}\big[\nabla_{\widehat{\bf x}} \widehat{u}_1-(\det {\rm D}\kappa)^{-1} \widehat{u}_1 \nabla_{\widehat{\bf x}}(\det {\rm D} \kappa)\big].
\ee
We conclude that $\norm{{\bf u}}{U} \eqsim \norm{\widehat {\bf u}}{\widehat U}$ and thus that $H$ is a linear isomorphism.

We set
$\widehat{\bf A}:={\bf A} \circ \kappa$,
$\widehat{\bf b}:={\bf b} \circ \kappa$, 
$\widehat{c}:=c \circ \kappa$,
and
\begin{align*}
\widetilde{\bf A}&:=({\rm D}_{\widehat{\bf x}}\kappa')^{-1} \widehat{\bf A} ({\rm D}_{\widehat{\bf x}}\kappa')^{-\top},\\
\widetilde{\bf b} &:=({\rm D}_{\widehat{\bf x}}\kappa')^{-1} \widehat{\bf b},\\
\widetilde{c}&:=(\det {\rm D}\kappa)^{-1}\big[\widehat{c} - (\det D\kappa)^{-1}({\rm D}_{\widehat{\bf x}}\kappa')^{-1} \widehat{\bf b} \cdot \nabla_{\widehat{\bf x}}(\det {\rm D} \kappa)\big]\\
\widehat{\bf w}&:= ({\rm D}_{\widehat{\bf x}}\kappa')^{-1}\partial_t \kappa' - (\det {\rm D}\kappa)^{-1} \widetilde{\bf A} \nabla_{\widehat{\bf x}}\det {\rm D} \kappa.
\end{align*}
Setting $\widehat{f}_1:=f_1 \circ \kappa$,
$\widehat{{\bf f}}_2:={\bf f}_2 \circ \kappa$,
$\widehat{u}_0:=u_0 \circ \kappa'(0,\cdot)$ and using \eqref{101}--\eqref{103}, we find that \eqref{eq:first-order system2} is equivalent to
\begin{align*}
\widehat{f}_1=(\divv {\bf u} +{\bf b}\cdot \nabla_\x u_1+ cu_1)\circ \kappa &= (\det {\rm D}\kappa)^{-1}
\big[\divv \widehat{\bf u}+\widetilde{\bf b} \cdot  \nabla_{\widehat{\bf x}} \widehat{u}_1+\widetilde{c} \,\widehat u_1\big],\\
\widehat{{\bf f}}_2=-({\bf u}_2 + {\bf A}\nabla_\x u_1)\circ \kappa &=
-(\det {\rm D}\kappa)^{-1}{\rm D}_{\widehat{\bf x}}\kappa'\big[\widehat{\bf u}_2+\widetilde{\bf A}\nabla_{\widehat{\bf x}} \widehat{u}_1+\widehat{u}_1 \widehat{\bf w}
\big],\\
\widehat{u}_0=(u_1 \circ \kappa)(0,\cdot)&=(\det {\rm D}_{\widehat{\bf x}} \kappa'(0,\cdot))^{-1} \widehat{u}_1(0,\cdot).
\end{align*}
Defining the linear isomorphism $F$ via 
\begin{align*}
F^{-1}{\bf f}:= \left(\begin{array}{@{}c@{}}  \det {\rm D}\kappa\,\widehat f_1 \\ \det {\rm D}\kappa\, ({\rm D}_{\widehat \x}\kappa')^{-1} \,\widehat{\bf f}_2\\ (\det {\rm D}_{\widehat{\bf x}} \kappa'(0,\cdot))\widehat u_0 \end{array} \right),
\end{align*}
it remains to show that
\begin{align*}
\widetilde G\widehat{\bf u}:=  \left(\begin{array}{@{}c@{}} \divv \widehat {\bf u} + \widetilde{\bf b}\cdot \nabla_{\widehat\x} \widehat u_1 + {\widetilde c} \,\widehat u_1 \\ - \widehat{\bf u}_2 - {\bf A}\nabla_{\widehat\x} \widehat u_1-\widehat u_1\widehat{\bf w} \\  \widehat u_1(0,\cdot)\end{array} \right)
\end{align*}
is  linear isomorphism from $\widehat U$ to $\widehat L$.

This follows from the identity 
\begin{align*}
\widetilde G \left(\begin{array}{@{}cc@{}} I & 0 \\ -\widehat{\bf w} & I \end{array} \right) \widehat {\bf u}=
\left(\begin{array}{@{}c@{}} \divv \widehat{\bf u} + (\widetilde{\bf b}-\widehat{\bf w})\cdot \nabla_{\widehat\x} \widehat u_1 + (\widetilde c-\divv_{\widehat\x} \widehat{\bf w}) \widehat u_1 \\ - \widehat{\bf u}_2 - {\bf A}\nabla_{\widehat\x} \widehat u_1 \\  \widehat u_1(0,\cdot)\end{array} \right).
\end{align*}
By the assumed regularity of $\kappa$ and $\kappa^{-1}$, \eqref{eq:Dkappa} yields that $\divv_{\widehat\x}\widehat {\bf w}\in L_\infty(\widehat Q)$, so that the latter mapping is a linear isomorphism from $\widehat U$ to $\widehat L$ according to Theorem~\ref{thm:fosls}.
Noting that $\left(\begin{array}{@{}cc@{}} I & 0 \\ -\widehat{\bf w} & I \end{array} \right)$ is a linear isomorphism from $\widehat U$ to $\widehat U$, we conclude the proof.
\end{proof}

\begin{remark}
Assume that also $\partial_t\partial_{\widehat x_i} \kappa'\in L_\infty(\widehat Q)$ for all $i\in\{1,\dots,d\}$. 
Given ${\bf f}=(f_1,{\bf f}_2,u_0)\in L$,  ${\bf u}=(u_1,{\bf u}_2)\in U$ then solves~\eqref{eq:first-order system2} if and only if $u_1 \in X:=\set{\widehat u\circ \kappa^{-1}}{\widehat u \in \widehat X}$ solves 
\begin{align*}
 \int_Q \partial_t u \,v+ ({\bf A}\nabla_\x u)\cdot\nabla_\x v + {\bf b}\cdot \nabla_\x u\,v + cuv \d \x \d t 
 &= \int_{Q} f_1 v + \divv_\x {\bf f}_2 \,v \d\x \d t \quad\text{for all }v\in Y,
\end{align*}
and  $u_1(0,\cdot) = u_0$.
Indeed, ${\bf u}=(u_1,{\bf u}_2)\in U$ together with \eqref{eq:Dkappa}, \eqref{102}, and Lemma~\ref{lem:U2X} imply that $\det {\rm D}_{\widehat\x}\kappa' \,u_1\circ\kappa  \in \widehat X$, and thus by the additional regularity assumption that $u_1\circ\kappa\in\widehat X$. 
The remainder follows from partial integration.  
In particular, also the second part of Theorem~\ref{thm:fosls} holds analogously for time-dependent domains.
Together with the open mapping theorem, this insight also allows to generalize Theorem~\ref{thm:ss09}. 
\end{remark}

 Theorem~\ref{thm:100} shows that also in the time-dependent domain case one can approximate the solution of the parabolic PDE by applying a Galerkin discretization to the variational problem $\langle G {\bf u},G {\bf v}\rangle_L=\langle {\bf f},G {\bf v}\rangle_L$ for all ${\bf v} \in U$.
Thinking of finite element discretizations, in particular for polytopal domains $Q$, this provides an attractive alternative to discretizations that require a transformation of the PDE to a time-independent domain, e.g., ALE time-stepping methods~\cite{shd01,gs17,sg20}.

\subsection{Numerical experiments}\label{sec:move numerics}
We consider the heat equation, i.e., ${\bf A} := {\bf Id}$, ${\bf b} := 0$, and $c:=0$ in \eqref{eq:parabolic pde}. 
Moreover, we set ${\bf f}_2:=0$ and ${\bf u}:=(u,-\nabla_\x u)$. 
We will approximate the latter function ${\bf u}\in U$ in the conforming  subspace of continuous piecewise affine functions $U^\delta$ on triangulations of the space-time cylinder $Q$ whose first component vanishes on the lateral boundary $\Sigma$. 
Using that $\{{\bf v}=(v_1,{\bf v}_2) \in H^1(Q)^{d+1}\colon v_1|_\Sigma=0\} \hookrightarrow U$, we know that  $\norm{{\bf u}-{\bf u}^\delta}{U}=  \OO({\rm dofs}^{-\frac{1}{d+1}})$ for sufficiently smooth $u$ and uniform mesh-refinement.


\begin{figure}[h]
\begin{tikzpicture}
\tikzmath{\c=4;}
\draw[line width = 0.25mm] (0,0) -- (0,\c) -- (\c/2,3*\c/4) -- (\c,\c) -- (\c,0) -- (\c/2,\c/4) -- cycle;
\end{tikzpicture}
\qquad\qquad
\begin{tikzpicture}
\tikzmath{\c=4;}
\draw[line width = 0.25mm] (0,0,0) -- (0,0,\c) -- (0,\c,\c) -- (0,\c,0) -- cycle;
\draw[line width = 0.25mm] (\c/2,\c/4,\c/4) -- (\c/2,\c/4,3*\c/4) -- (\c/2,3*\c/4,3*\c/4) -- (\c/2,3*\c/4,\c/4) -- cycle;
\draw[line width = 0.25mm] (\c,0,0) -- (\c,0,\c) -- (\c,\c,\c) -- (\c,\c,0) -- cycle;
\draw[line width = 0.25mm] (0,0,0) -- (\c/2,\c/4,\c/4);
\draw[line width = 0.25mm] (0,0,\c) -- (\c/2,\c/4,3*\c/4);
\draw[line width = 0.25mm] (0,\c,\c) -- (\c/2,3*\c/4,3*\c/4);
\draw[line width = 0.25mm] (0,\c,0) -- (\c/2,3*\c/4,\c/4);
\draw[line width = 0.25mm] (\c/2,\c/4,\c/4) -- (\c,0,0);
\draw[line width = 0.25mm] (\c/2,\c/4,3*\c/4) -- (\c,0,\c);
\draw[line width = 0.25mm] (\c/2,3*\c/4,3*\c/4) -- (\c,\c,\c);
\draw[line width = 0.25mm] (\c/2,3*\c/4,\c/4) -- (\c,\c,0);
\end{tikzpicture}
\caption{Space-time domains $Q$ of Section~\ref{sec:move 1d} (left) and Section~\ref{sec:move 2d} (right).}\label{fig:geometries}
\end{figure}
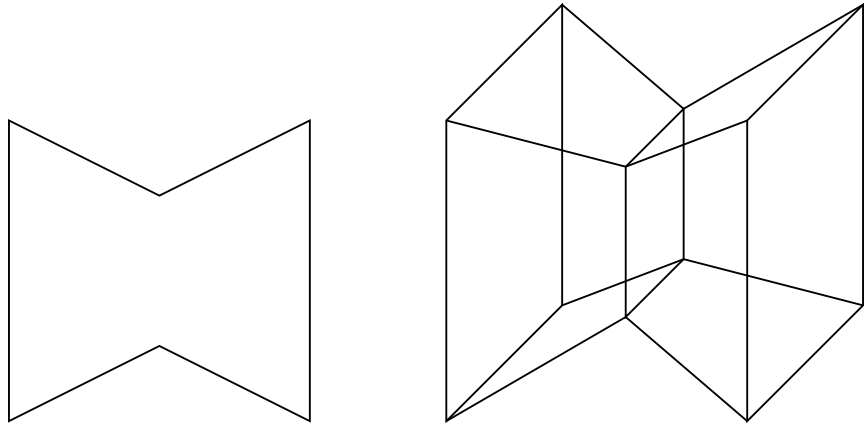

\subsubsection{Experiment in 1+1D}\label{sec:move 1d}
Let 
\begin{align*}
 Q := {\rm int}\Bigg(\underbrace{{\rm conv}\left\{\begin{pmatrix}0\\0\end{pmatrix}, \begin{pmatrix}0.5\\0.25\end{pmatrix}, \begin{pmatrix}0.5\\0.75\end{pmatrix}, \begin{pmatrix}0\\1\end{pmatrix} \right\}}_{=: Q_1} 
 \cup \underbrace{{\rm conv}\left\{\begin{pmatrix}0.5\\0.25\end{pmatrix}, \begin{pmatrix}1\\0\end{pmatrix}, \begin{pmatrix}1\\1\end{pmatrix}, \begin{pmatrix}0.5\\0.75\end{pmatrix} \right\}}_{=: Q_2}\Bigg),
\end{align*}
where ${\rm int}(\cdot)$ denotes the interior of a set and ${\rm conv}(\cdot)$ denotes the convex hull of a set of points; see Figure~\ref{fig:geometries} for an illustration.
With the bijection $\kappa\colon[0,1]^2\to \overline Q\colon (t,\widehat{x})\mapsto \big(t, (\widehat{x}-\tfrac12) (|t-\tfrac12|+\tfrac12) + \tfrac12 \big)$, we prescribe the solution $u$ by
\begin{align*}
 u(\kappa(t,\widehat{x})) := \sin(\pi \widehat x),
\end{align*}
and choose $f_1:=\partial_t u - \Delta_x u$, and $u_0:=u(0,\cdot)$. 

Starting on an initial triangulation $\TT^\delta$ of $Q$ with six elements (such that each element is either in $\overline Q_1$ or $\overline Q_2$), we define a sequence of triangulations $\TT^\delta$ by splitting  all elements in the previous $\TT^\delta$ into four new elements by repeated newest vertex bisection.
We compute the Galerkin approximation ${\bf u}^\delta\in U^\delta$ of ${\bf u}\in U$ with respect to the scalar product $\dual{G(\cdot)}{G(\cdot)}_L$, with $U^\delta$ being the $2$-fold Cartesian product of continuous piecewise affine functions on $\TT^\delta$ whose first component vanishes on $\Sigma$.
The corresponding convergence plot is displayed in Figure~\ref{fig:move 1d}. 
Here, $u^\delta$ denotes the first component of ${\bf u}^\delta$. 
While $\norm{{G\bf u}-G{\bf u^\delta}}{L}\simeq\norm{{\bf u}-{\bf u^\delta}}{U}$ and $\norm{\nabla_\x (u-u^\delta)}{L_2(Q)}$ converge as expected at rate $\OO({\rm dofs}^{-\frac12})$, 
$\norm{u(0,\cdot) -u^\delta(0,\cdot)}{L_2(\Omega)}$, $\norm{u(T,\cdot) -u^\delta(T,\cdot)}{L_2(\Omega)}$, and $\norm{u -u^\delta}{L_2(Q)}$ converge roughly with the rate $\OO({\rm dofs}^{-0.8})$.

\begin{center}
\begin{figure}
\begin{tikzpicture}
\begin{loglogaxis}[
width = 0.75\textwidth,
clip = true,
xlabel= {degrees of freedom},
grid = major,
legend pos = south west,
legend entries = {$\norm{G{\bf u} - G{\bf u}^\delta}{L}$\\$\norm{\nabla_\x (u-u^\delta)}{L_2(Q)}$\\ $\norm{u(0,\cdot) -u^\delta(0,\cdot)}{L_2(\Omega)}$ \\ $\norm{u(T,\cdot) -u^\delta(T,\cdot)}{L_2(\Omega)}$ \\ $\norm{u -u^\delta}{L_2(Q)}$ \\} 
]
\addplot[line width = 0.25mm,color=blue,mark=triangle*] table[x=ndof,y=est, col sep=comma] {data/stmove_D1_order1_theta100.csv};
\addplot[line width = 0.25mm,color=red,mark=*] table[x=ndof,y=err_Y, col sep=comma] {data/stmove_D1_order1_theta100.csv};
\addplot[line width = 0.25mm,color=green,mark=x] table[x=ndof,y=err_0, col sep=comma] {data/stmove_D1_order1_theta100.csv};
\addplot[line width = 0.25mm,color=magenta,mark=+] table[x=ndof,y=err_T, col sep=comma] {data/stmove_D1_order1_theta100.csv};
\addplot[line width = 0.25mm,color=cyan,mark=asterisk] table[x=ndof,y=err_l2, col sep=comma] {data/stmove_D1_order1_theta100.csv};

\tikzmath{\xc=100000;\yc=0.15;\slope=1/2;}
\draw[line width = 0.25mm] (axis cs:\xc,\yc) -- (axis cs:10*\xc,\yc) -- (axis cs:10*\xc,\yc * 0.1^\slope) -- cycle;
\node[above] at (axis cs:\xc*10^0.5,\yc) {$1$};
\node[right] at (axis cs:10*\xc,\yc * 0.1^0.5^\slope) {$\frac12$};

\tikzmath{\xcc=100000;\ycc=0.00023;\slopee=0.8;}
\draw[line width = 0.25mm] (axis cs:\xcc,\ycc) -- (axis cs:\xcc,\ycc * 0.1^\slopee) -- (axis cs:10*\xcc,\ycc * 0.1^\slopee) -- cycle;
\node[left] at (axis cs:\xcc,\ycc * 0.1^0.5^\slopee) {$0.8$};
\node[below] at (axis cs:\xcc*10^0.5,\ycc * 0.1^\slopee) {$1$};

\end{loglogaxis}
\end{tikzpicture}
\caption{\label{fig:move 1d} Time-dependent domain problem in 1+1D of Section~\ref{sec:move 1d}.}
\end{figure}
\end{center}

\subsubsection{Experiment in 2+1D}\label{sec:move 2d}
Let 
\begin{align*}
 Q := {\rm int} \Bigg(&\underbrace{{\rm conv}\left\{\begin{pmatrix}0\\0\\0\end{pmatrix}, \begin{pmatrix}0\\1\\0\end{pmatrix}, \begin{pmatrix}0\\1\\1\end{pmatrix}, \begin{pmatrix}0\\0\\1\end{pmatrix},
 \begin{pmatrix}0.5\\0.25\\0.25\end{pmatrix}, \begin{pmatrix}0.5\\0.75\\0.25\end{pmatrix}, \begin{pmatrix}0.5\\0.75\\0.75\end{pmatrix}, \begin{pmatrix}0.5\\0.25\\0.75\end{pmatrix}\right\}}_{=:Q_1},
 \\
 \cup &\underbrace{{\rm conv}\left\{ \begin{pmatrix}0.5\\0.25\\0.25\end{pmatrix}, \begin{pmatrix}0.5\\0.75\\0.25\end{pmatrix}, \begin{pmatrix}0.5\\0.75\\0.75\end{pmatrix}, \begin{pmatrix}0.5\\0.25\\0.75\end{pmatrix},
 \begin{pmatrix}1\\0\\0\end{pmatrix}, \begin{pmatrix}1\\1\\0\end{pmatrix}, \begin{pmatrix}1\\1\\1\end{pmatrix}, \begin{pmatrix}1\\0\\1\end{pmatrix}\right\}}_{=:Q_2}\Bigg).
\end{align*}
where ${\rm int}(\cdot)$ denotes the interior of a set and ${\rm conv}(\cdot)$ denotes the convex hull of a set of points; see Figure~\ref{fig:geometries} for an illustration.
With the bijection $\kappa\colon[0,1]^3\to \overline Q\colon (t,\widehat{\bf x})\mapsto \big(t, \big(\widehat{\bf x}-(\tfrac12,\tfrac12)\big) (|t-\tfrac12|+\tfrac12) + (\tfrac12,\tfrac12) \big)$, we prescribe the solution $u$ by
\begin{align*}
 u(\kappa(t,\widehat{x}_1,\widehat{x}_2)) := \sin(\pi \widehat{x}_1) \sin(\pi \widehat{x}_2),
\end{align*} 
and choose $f_1:=\partial_t u - \Delta_{\bf x} u$, and $u_0:=u(0,\cdot)$.

Starting on an initial triangulation $\TT^\delta$ of $Q$ with twelve elements (such that each element is either in $\overline Q_1$ or $\overline Q_2$), we define a sequence of triangulations $\TT^\delta$ by splitting  all elements in the previous $\TT^\delta$ into eight new elements by repeated newest vertex bisection.
We compute the Galerkin approximation ${\bf u}^\delta\in U^\delta$ of ${\bf u}\in U$ with respect to the scalar product $\dual{G(\cdot)}{G(\cdot)}_L$, with $U^\delta$ being the $3$-fold Cartesian product of continuous piecewise affine functions on $\TT^\delta$ whose first component vanishes on $\Sigma$.
The corresponding convergence plot is displayed in Figure~\ref{fig:move 2d}. 
Here, $u^\delta$ denotes the first component of ${\bf u}^\delta$. 
While $\norm{{G\bf u}-G{\bf u^\delta}}{L}\simeq\norm{{\bf u}-{\bf u^\delta}}{U}$ and $\norm{\nabla_\x (u-u^\delta)}{L_2(Q)}$ converge as expected at rate $\OO({\rm dofs}^{-\frac13})$, 
$\norm{u(0,\cdot) -u^\delta(0,\cdot)}{L_2(\Omega)}$, $\norm{u(T,\cdot) -u^\delta(T,\cdot)}{L_2(\Omega)}$, and $\norm{u -u^\delta}{L_2(Q)}$ converge roughly with the rate $\OO({\rm dofs}^{-0.6})$. 

\begin{center}
\begin{figure}
\begin{tikzpicture}
\begin{loglogaxis}[
width = 0.75\textwidth,
clip = true,
xlabel= {degrees of freedom},
grid = major,
legend pos = south west,
legend entries = {$\norm{G{\bf u} - G{\bf u}^\delta}{L}$\\$\norm{\nabla_\x (u-u^\delta)}{L_2(Q)}$\\ $\norm{u(0,\cdot) -u^\delta(0,\cdot)}{L_2(\Omega)}$ \\ $\norm{u(T,\cdot) -u^\delta(T,\cdot)}{L_2(\Omega)}$ \\ $\norm{u -u^\delta}{L_2(Q)}$\\} 
]
\addplot[line width = 0.25mm,color=blue,mark=triangle*] table[x=ndof,y=est, col sep=comma] {data/stmove_D2_order1_theta100.csv};
\addplot[line width = 0.25mm,color=red,mark=*] table[x=ndof,y=err_Y, col sep=comma] {data/stmove_D2_order1_theta100.csv};
\addplot[line width = 0.25mm,color=green,mark=x] table[x=ndof,y=err_0, col sep=comma] {data/stmove_D2_order1_theta100.csv};
\addplot[line width = 0.25mm,color=magenta,mark=+] table[x=ndof,y=err_T, col sep=comma] {data/stmove_D2_order1_theta100.csv};
\addplot[line width = 0.25mm,color=cyan,mark=asterisk] table[x=ndof,y=err_l2, col sep=comma] {data/stmove_D2_order1_theta100.csv};

\tikzmath{\xc=15000;\yc=2;\slope=1/3;}
\draw[line width = 0.25mm] (axis cs:\xc,\yc) -- (axis cs:10*\xc,\yc) -- (axis cs:10*\xc,\yc * 0.1^\slope) -- cycle;
\node[above] at (axis cs:\xc*10^0.5,\yc) {$1$};
\node[right] at (axis cs:10*\xc,\yc * 0.1^0.5^\slope) {$\frac13$};

\tikzmath{\xcc=15000;\ycc=0.009;\slopee=0.6;}
\draw[line width = 0.25mm] (axis cs:\xcc,\ycc) -- (axis cs:\xcc,\ycc * 0.1^\slopee) -- (axis cs:10*\xcc,\ycc * 0.1^\slopee) -- cycle;
\node[left] at (axis cs:\xcc,\ycc * 0.1^0.5^\slopee) {$0.6$};
\node[below] at (axis cs:\xcc*10^0.5,\ycc * 0.1^\slopee) {$1$};

\end{loglogaxis}
\end{tikzpicture}
\caption{\label{fig:move 2d} Time-dependent domain problem in 2+1D of Section~\ref{sec:move 2d}.}
\end{figure}
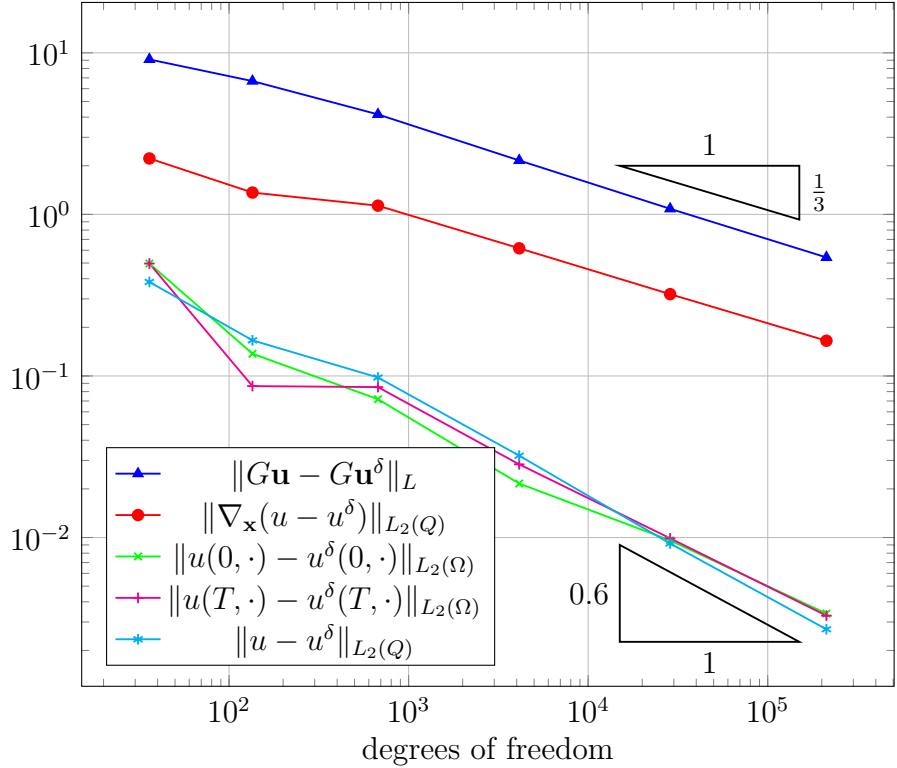
\end{center}

\section{Conclusion}\label{sec:conclusion}
In this work, we have demonstrated that the space-time FOSLS~\cite{fk21} and its generalization~\cite{gs21} to general second-order parabolic PDEs can be easily applied to solve parameter-dependent problems, optimal control problems, and time-dependent domain problems. 
In each case, completely unstructured space-time finite elements may be employed, and our numerical experiments exhibit optimal convergence. 
We also want to stress that the FOSLS is applicable to the combined problem, i.e., parameter-dependent optimal control problems on time-dependent domains.  
Indeed, the inf-sup stability of the saddle-point problem corresponding to an optimal control problem essentially only hinges on the coercivity of the bilinear form $a$ (see Lemma~\ref{lem:optimal control}), which is also valid in the case of time-dependent domains (see Theorem~\ref{thm:100}).
As this stability holds uniformly for \emph{arbitrary} trial spaces, a reduced basis method as in Section~\ref{sec:reduced basis} can be employed, where the greedy algorithm of Section~\ref{sec:greedy algorithm} can be steered by one of the estimators discussed in Section~\ref{sec:a posteriori}.

\bibliographystyle{alpha}
\bibliography{literature}

\end{document}